\documentclass[a4paper,11pt]{amsart}
\usepackage{amssymb, enumitem}
\usepackage{tikz}
\usepackage[all]{xy}

\usepackage{hyperref, aliascnt}
\usepackage{mathtools}
\usepackage{xcolor}
\def\today{\number\day\space\ifcase\month\or   January\or February\or
   March\or April\or May\or June\or   July\or August\or September\or
   October\or November\or December\fi\   \number\year}

\theoremstyle{plain}
\newtheorem{lma}{Lemma}[section]

\newaliascnt{thmCt}{lma}
\newtheorem{thm}[thmCt]{Theorem}
\aliascntresetthe{thmCt}

\newaliascnt{corCt}{lma}
\newtheorem{cor}[corCt]{Corollary}
\aliascntresetthe{corCt}

\newaliascnt{propCt}{lma}
\newtheorem{prop}[propCt]{Proposition}
\aliascntresetthe{propCt}

\newtheorem*{thm*}{Theorem}
\newtheorem*{cor*}{Corollary}
\newtheorem*{prop*}{Proposition}

\newcounter{theoremintro}
\newtheorem{thmintro}[theoremintro]{Theorem}
\newtheorem{corintro}[theoremintro]{Corollary}
\theoremstyle{definition}
\newtheorem{quesintro}[theoremintro]{Question}

\newaliascnt{pgrCt}{lma}

\aliascntresetthe{pgrCt}

\newaliascnt{dfCt}{lma}
\newtheorem{df}[dfCt]{Definition}
\aliascntresetthe{dfCt}

\newaliascnt{remCt}{lma}
\newtheorem{rem}[remCt]{Remark}
\aliascntresetthe{remCt}

\newaliascnt{remsCt}{lma}

\aliascntresetthe{remsCt}

\newaliascnt{egCt}{lma}

\aliascntresetthe{egCt}

\newaliascnt{egsCt}{lma}

\aliascntresetthe{egsCt}

\newaliascnt{qstCt}{lma}

\aliascntresetthe{qstCt}

\newaliascnt{pbmCt}{lma}

\aliascntresetthe{pbmCt}

\newaliascnt{notaCt}{lma}

\aliascntresetthe{notaCt}

\theoremstyle{theorem}

\newcommand{\beq}{\begin{equation}}
\newcommand{\eeq}{\end{equation}}
\newcommand{\beqa}{\begin{eqnarray*}}
\newcommand{\eeqa}{\end{eqnarray*}}
\newcommand{\bal}{\begin{align*}}
\newcommand{\eal}{\end{align*}}
\newcommand{\bi}{\begin{itemize}}
\newcommand{\ei}{\end{itemize}}
\newcommand{\be}{\begin{enumerate}}
\newcommand{\ee}{\end{enumerate}}

\newcommand{\cU}{\mathcal{U}}

\newcommand{\ep}{\varepsilon}

\newcommand{\N}{{\mathbb{N}}}

\newcommand{\K}{{\mathcal{K}}}

\newcommand{\U}{{\mathcal{U}}}

\numberwithin{equation}{section}

\newcommand{\Aut}{{\mathrm{Aut}}}

\newcommand{\rk}{\mathrm{rk}}
\newcommand{\cstar}{$C^*$}

\newcommand{\I}{\infty}

\title[Dynamical comparison and $\mathcal{Z}$-stability for crossed products]{Dynamical comparison and $\mathcal{Z}$-stability for crossed products of simple $C^*$-algebras}

\thanks{
The first named author was supported by a starting grant of 
the Swedish Research Council.
The second named author was supported by the project G085020N 
funded by the Research Foundation Flanders (FWO), and by the generosity of Eric and Wendy Schmidt by recommendation of the Schmidt Futures program. The third named author was partially supported by the Deutsche Forschungsgemeinschaft (DFG, German Research Foundation) under Germany's Excellence Strategy ``EXC 2044'' 390685587, Mathematics M\"unster ``Dynamics, Geometry and Structure'', and under Project-ID 427320536, SFB 1442; and ERC Advanced Grant 834267 - AMAREC.
The fourth named author is supported by European Union's Horizon 2020 research and innovation programme under the Marie Sk\l{}odowska-Curie grant agreement No.~891709. 
}

\author[Gardella]{Eusebio Gardella}
\address[Eusebio Gardella]
{Department of Mathematical Sciences, Chalmers University of
Technology and University of Gothenburg, Gothenburg SE-412 96, Sweden.}
\email{gardella@chalmers.se}
\urladdr{www.math.chalmers.se/~gardella}

\author[Geffen]{Shirly Geffen}
\address[Shirly Geffen]
{Department of Mathematics, KU Leuven, Celestijnenlaan 200B, 3001 Leuven, Belgium.}
\email{shirly.geffen@kuleuven.be}
\urladdr{https://shirlygeffen.com/}

\author[Naryshkin]{Petr Naryshkin}
\address[Petr Naryshkin]
{Mathematisches Institut, Fachbereich Mathematik und Informatik der
Universit\"at M\"unster, Einsteinstrasse 62, 48149 M\"unster, Germany.}
\email{pnaryshk@uni-muenster.de}
\urladdr{http://petrnaryshkin.wordpress.com}

\author[Vaccaro]{Andrea Vaccaro}
\address[Andrea Vaccaro]
{Institut de Math\'ematiques de Jussieu (IMJ-PRG)\\
Universit\'e de Paris\\
B\^atiment Sophie Germain\\
8 Place Aur\'elie Nemours \\ 75013 Paris, France.}
\email{vaccaro@imj-prg.fr}
\urladdr{https://sites.google.com/view/avaccaro/home}
 
\begin{document}

\begin{abstract}
We establish $\mathcal{Z}$-stability for crossed products
of outer actions of amenable groups on $\mathcal{Z}$-stable $C^*$-algebras 
under a mild technical assumption which we call \emph{McDuff property with respect to invariant traces}.
We obtain such result using a weak form of dynamical comparison, which we verify in great generality. We complement our results by proving that McDuffness with respect to invariant traces is automatic in many cases of interest. This is the case, for instance, for every action of an amenable group $G$
on a classifiable $C^*$-algebra $A$ whose trace space $T(A)$ is a Bauer simplex with finite dimensional boundary $\partial_e T(A)$, and such that the induced action $G \curvearrowright \partial_eT(A)$ is free. If $G = \mathbb{Z}^d$ and the action
$G \curvearrowright \partial_eT(A)$ is free and minimal, then we obtain McDuffness with respect to invariant traces, and thus $\mathcal{Z}$-stability of the corresponding crossed product, also in case $\partial_e T(A)$ has infinite covering dimension.

\end{abstract} 

\maketitle

\tableofcontents

\renewcommand*{\thetheoremintro}{\Alph{theoremintro}}

\section{Introduction}

One of the highlights in the development of $C^*$-algebra theory in the last 
couple of decades
has been the tremendous progress in the Elliott classification 
programme. 
There is now a very satisfactory and nearly final classification theorem,
which in this form is the combination of the Kirchberg-Phillips
classification of purely infinite $C^*$-algebras \cite{Phi_classification_2000}, Elliott-Gong-Lin-Niu's classification
of simple $C^*$-algebras of finite nuclear dimension \cite{EGLN},
and the Tikuisis-White-Winter quasidiagonality theorem 
\cite{TikWhiWin_quasidiagonality_2017} (these results build upon a large body of previous works; we refer to \cite{winter18}
for a complete bibliography on the topic). We reproduce the classification 
theorem below:

\begin{thm*}
Let $A$ and $B$ be simple, separable, unital, nuclear, $\mathcal Z$-stable 
$C^*$-algebras satisfying the Universal Coefficient Theorem (UCT).
Then $A\cong B$ if and only if $\mathrm{Ell}(A)\cong \mathrm{Ell}(B)$.
\end{thm*}

The Elliott invariant $\mathrm{Ell}(A)$ of a $C^*$-algebra $A$ consists
of its $K$-theory and trace space (together with the natural
pairing between them). The theorem above is virtually optimal: except for unitality and 
for the possibly redundant requirement on the UCT, all other conditions
on the $C^*$-algebras are needed in order to obtain a classification
result in terms of the Elliott invariant. We will say that a $C^*$-algebra is \emph{classifiable} if it satisfies the assumptions of the classification
theorem above.

The present work is concerned with group actions on classifiable 
$C^*$-algebras, and we are interested in knowing when the classifiable
class is closed under formation of reduced crossed products by countable,
discrete groups. Explicitly,
let $A$ be a classifiable $C^*$-algebra, let $G$ be a countable, discrete,
amenable\footnote{This question is also interesting for nonamenable groups, although
a completely different set of techniques seems to be needed to analyze it 
in that setting; see \cite{GGKNV_tracial_2022} for some recent results in this direction.} group, and let $\alpha\colon G\to\Aut(A)$ be an action. 
When is $A\rtimes_\alpha G$ again classifiable?

Some of the conditions listed in the classification theorem are automatically
satisfied by $A\rtimes_\alpha G$; this is the case for unitality, separability, and nuclearity.
Simplicity is preserved whenever $\alpha_g$ is outer for all $g\in G\setminus\{e\}$, as proved by Kishimoto in \cite{KishOut}. 
Preservation of the UCT is somewhat more subtle: it is automatic if $G$ is
torsion free, but already for finite groups, preservation of the UCT is
equivalent to the UCT problem \cite{BarSza_cartanII_2020}. 
In particular, we have a good understanding of the preservation of all properties except for 
$\mathcal{Z}$-stability, which is the main focus of the present work. 

\begin{quesintro}\label{qstA}
Let $A$ be a simple, separable, unital, nuclear, $\mathcal Z$-stable 
$C^*$-algebra, let $G$ be a countable, discrete,
amenable group, and let $\alpha\colon G\to\Aut(A)$ be an outer action. 
When is $A\rtimes_\alpha G$ again $\mathcal{Z}$-stable?
\end{quesintro}

Partial results addressing this question have been obtained by a number of 
authors \cite{HO, szabo:equivariantKP, matuisato:1, matuisato:2,
sato, szabo:equiSI, HW07, HWZ15, SWZ19, G, GL, GHV, Wouters, AminiGolestaniJamaliPhillips2022}. 

Closely related to this topic is \cite[Conjecture A]{szabo:equiSI}, where it has been conjectured that an action 
as in Question~\autoref{qstA} ought to automatically tensorially
absorb the identity on $\mathcal{Z}$ up to cocycle equivalence; this property is 
often referred to as \emph{equivariant $\mathcal{Z}$-stability} for 
$\alpha$. It is immediate to check that the crossed product of an 
equivariantly $\mathcal{Z}$-stable action is 
$\mathcal{Z}$-stable. As a consequence, confirming such conjecture would imply that 
$\mathcal{Z}$-stability is automatically preserved for crossed products as in Question~\autoref{qstA}. Indeed, in most cases where crossed products have been shown to be $\mathcal{Z}$-stable, this was obtained by directly
proving equivariant $\mathcal{Z}$-stability of the action.

More precisely, equivariant $\mathcal{Z}$-stability has been verified in full generality for actions as in Question~\autoref{qstA} in case the $C^*$-algebra $A$
is purely infinite \cite{szabo:equivariantKP}. In the stably finite case, Matui and Sato's seminal work \cite{matuisato:1,
matuisato:2} started a rich literature of partial results (see \cite{sato, szabo:equiSI, GHV, Wouters}) all relying on the additional assumptions that the
extreme boundary $\partial_eT(A)$ of the tracial states space $T(A)$ is compact and 
finite-dimensional. It seems that the technology currently available
is not sufficient to push these results any further, and we have been so far unable to establish
equivariant $\mathcal{Z}$-stability whenever
$\partial_eT(A)$ fails to be finite-dimensional and compact. 

We overcome this obstacle by introducing the notion of \emph{McDuff property with respect to invariant traces} (\autoref{def:wuMc}), a technical condition
for actions on $C^*$-algebras which allows us to obtain $\mathcal{Z}$-stability of 
$A\rtimes_\alpha G$ in situations where equivariant $\mathcal{Z}$-stability
is far from being established.

Our definition originates from the
\emph{uniform McDuff property} (\autoref{df:McDuff}). This notion was formally introduced in \cite[Definition 4.2]{cetw}, but it
implicitly appeared in numerous earlier papers where $\mathcal{Z}$-stability was derived from strict comparison in the framework of the Toms--Winter Conjecture, starting with Matui and Sato's work
\cite{ms:stric} (see also \cite{tww, KirchRor_centr, lin:trosc}). Broadly speaking,
a unital $C^*$-algebra $A$ has the uniform McDuff property if, for any $d \in \mathbb{N}$, there exist completely positive order zero maps
$\varphi\colon M_d \to A$ with approximately central range and such that the remainder $1 - \varphi(1)$ is arbitrarily small with respect
to all traces of $A$.
An equivariant version of this definition, requiring in addition that the ranges of these maps are approximately invariant,
has proven to be an invaluable tool in showing equivariant $\mathcal{Z}$-stability,
and its (sometimes implicit) presence can be traced in most works related to the topic \cite{matuisato:1, matuisato:2, sato, GHV, Wouters}. As we mentioned,
the main issue with this equivariant analogue of uniform McDuffness is that, with the theory that we currently have at our disposal, we can only verify it for actions on $C^*$-algebras whose trace space is a Bauer simplex with finite-dimensional boundary.

Given a $C^*$-algebra $A$ with $T(A)\neq \emptyset$, and an action $\alpha \colon G \to \text{Aut}(A)$, the definition of McDuff property with respect to invariant traces requires the
existence of completely positive order zero maps $\varphi\colon M_d \to A$ with approximately central and approximately invariant range,
similarly to the aforementioned equivariant McDuffness. However, it differs from the latter in that it only prescribes for $1 - \varphi(1)$ to be small with respect to invariant traces, rather than all of them. This distinction is not a 
minor one, as it makes McDuffness with respect to invariant traces easier to check (see \autoref{s.URP}, Corollary \ref{cor:E} and Corollary~\autoref{cor:Z}), but at the same time it produces a tool specifically suited
to verify $\mathcal{Z}$-stability of 
$A\rtimes_\alpha G$, rather than equivariant $\mathcal{Z}$-stability of $\alpha$, as clear from our first main result.

\begin{thmintro}[\autoref{thm:main2}] \label{thmintro.B}
Let $G$ be a countable, discrete, amenable group, let $A$ be a simple, separable, stably finite, unital, nuclear, $\mathcal{Z}$-stable
\cstar-algebra, and 
let $\alpha \colon G \to \mathrm{Aut}(A)$ be an outer 
action such that $(A,\alpha)$ has the McDuff property with respect to invariant traces. 
Then $A \rtimes_\alpha G$ is a simple,
separable, unital, nuclear, $\mathcal{Z}$-stable
$C^*$-algebra. 
\end{thmintro}



Our strategy is to show that $A\rtimes_\alpha G$ is \emph{tracially
$\mathcal{Z}$-stable} (see \autoref{def:trZst}), and then take advantage of nuclearity to conclude $\mathcal{Z}$-stability
\cite[Theorem 4.1]{HO}. Tracial $\mathcal{Z}$-stability requires the existence of completely positive order zero maps
$\varphi\colon M_d \to A$ with approximately central range and such that the value $1 - \varphi(1)$ is suitably small in the sense of
Cuntz subequivalence. If $A\rtimes_\alpha G$ is a crossed product as in the statement of Theorem~\autoref{thmintro.B}, then every trace on
it restricts to some element in $T(A)^\alpha$, the set of $\alpha$-invariant traces on $A$. Thanks to this observation,
it is possible to check that if $(A,\alpha)$ has the McDuff property with respect to invariant traces, then $A \rtimes_\alpha G$ is uniformly McDuff (with maps
whose image are contained in $A$).
This does not quite imply tracial $\mathcal{Z}$-stability of $A \rtimes_\alpha G$, as it only gives almost central order zero maps $M_d \to A \rtimes_\alpha G$
with tracially large range, rather than large in the sense of Cuntz subequivalence (i.e.\ as in item \ref{item:Z2} of \autoref{def:trZst}).

In order to fill this gap, a comparison argument is needed.
This led us to the notion of \emph{weak dynamical comparison} (\autoref{df:wdc}), a key property inspired by the upcoming work \cite{BPWZ_dynamical_2022}, which allows to
compare (in the Cuntz sense) elements from $A$ inside of 
$A\rtimes_\alpha G$\footnote{Of course it would suffice to compare them in $A$, but if we could do 
this, we would be able to prove equivariant $\mathcal{Z}$-stability of 
$\alpha$ and 
not just $\mathcal{Z}$-stability of $A\rtimes_\alpha G$.} only by
knowing that they compare on invariant traces. 
The fact that weak dynamical comparison holds in great generality is our other main result, and we 
believe it will see other applications in the 
future: 

\begin{thmintro}[{\autoref{cor:thmB}}]\label{thmB}
Let $G$ be a countable, discrete, amenable group and let $A$ be a simple, separable, unital, exact, $\mathcal{Z}$-stable \cstar-algebra.
Then every action of $G$ on $A$ has \emph{weak dynamical comparison}. In particular, this implies that
if $a,b\in (A\otimes \mathcal{K})_+$ with $b\neq 0$ satisfy
\[\max_{\tau\in T(A)^\alpha} d_\tau(a) < 
 \min_{\tau\in T(A)^\alpha} d_\tau(b),
\]
then $a\precsim b$ in $A\rtimes_\alpha G$.
 \end{thmintro}

We complement these results by establishing McDuffness with respect to invariant traces in cases
of interest not covered in the previous literature in \cite{GHV, Wouters}.
To this end, we introduce the \emph{weak Rokhlin property with respect to invariant traces}, which is a weakening of the weak Rokhlin property considered in 
\cite{Arc_crossed_2008, HO, GarHirSan_rokhlin_2017, GHV}
in that the 
left-over of the towers is only assumed to be small with respect to all \emph{invariant}
traces, and not with respect to all traces (or in the sense of Cuntz 
comparison). We show in \autoref{lemma:cpcoz} that a dynamical system $(A,\alpha)$ verifying such Rokhlin-type condition satisfies the McDuff property with respect to invariant traces, provided that $A$ is uniformly McDuff. This is done via an averaging argument
over the elements of the towers. By Theorem~\autoref{thmintro.B},
it follows that the corresponding crossed
product is $\mathcal{Z}$-stable.

In \autoref{s.URP} we present some concrete instances where the weak Rokhlin property with respect to invariant traces is realized. We do so by 
taking advantage of some notions and ideas coming from
a line of research that has run parallel to the progress made in the 
setting of the Elliott programme, and has focused on the study of 
crossed products arising from topological dynamical systems. Partly motivated by the problem of
finding when such crossed products satisfy the assumptions of the 
classification theorem, this theory has recently seen some impressive breakthroughs, obtaining positive results even
for actions on topological spaces with infinite covering dimension
\cite{Ker_dimension_2020, KerSza_almost_2020, KerNar_elementary_2021, Nar_polynomial_2021, GGKN_paradoxical_2022,Niu_comparison_2019}.

Thanks to Ozawa's groundbreaking work in \cite{ozawa_dix}, if $(A,\alpha)$ is a dynamical system as in Question~\autoref{qstA}
such that $T(A)$ is a non-empty Bauer simplex, then the induced action on $\partial_eT(A)$ can be described, up to some error,
using approximately central elements in $A$. This fact permits to apply some of the techniques and results developed
for dynamical systems in the aforementioned papers, to actions on simple $C^*$-algebras. In particular, we show that if the induced action on $\partial_eT(A)$ has the so-called \emph{uniform Rokhlin Property} (\autoref{def:URP};
see also {\cite[Definition~3.1]{Niu_comparison_2022}), then $(A,\alpha)$ has the weak Rokhlin property with respect to invariant traces, and therefore the resulting crossed product is $\mathcal{Z}$-stable.


\begin{corintro}[\autoref{cor:uRp}]\label{corintro:uRp}
Let $G$ be a countable, discrete, amenable group, let $A$ be a simple,
separable, unital, nuclear, $\mathcal{Z}$-stable
\cstar-algebra, such that $T(A)$ is a non-empty Bauer simplex, and 
let $\alpha \colon G \to \mathrm{Aut}(A)$ be an action such
that $G \curvearrowright \partial_e T(A)$ has the uniform Rokhlin property. Then $A \rtimes_\alpha G$ is a simple,
separable, unital, nuclear, $\mathcal{Z}$-stable
$C^*$-algebra. If moreover $A \rtimes_\alpha G$ satisfies the UCT, then it is a classifiable
$C^*$-algebra.

\end{corintro}

The uniform Rokhlin property follows from the \emph{small boundary property} \cite[Definition~5.1]{KerSza_almost_2020}, which automatically holds for all free actions of amenable groups on compact, metrizable, finite-dimensional spaces \cite[Theorem~3.8]{szaboRokDimZm}.
This fact, combined with Corollary~\autoref{corintro:uRp}, leads to the following corollary.

\begin{corintro} \label{cor:E}
Let $G$ be a countable, discrete, amenable group, let $A$ be a simple,
separable, unital, nuclear, $\mathcal{Z}$-stable
\cstar-algebra, such that $T(A)$ is a non-empty Bauer simplex and $\mathrm{dim}(\partial_eT(A)) < \infty$. Let $\alpha \colon G \to \mathrm{Aut}(A)$ be an action such that $G \curvearrowright \partial_e T(A)$ is free.  Then $A \rtimes_\alpha G$ is a simple,
separable, unital, nuclear, $\mathcal{Z}$-stable $C^*$-algebra.
\end{corintro}

Note that Corollary \ref{cor:E} yields a weaker conclusion of \cite[Theorem B]{Wouters}, as it does not grant equivariant $\mathcal{Z}$-stability of the action. However, \cite[Theorem B]{Wouters} is restricted to groups with polynomial growth, while we cover all amenable groups.

The uniform Rokhlin property is conjectured to hold for all free, minimal actions of amenable groups, and is verified in \cite[Theorem~4.2]{Niu_comparison_2019} (see also \cite{lindenMeandim, lindenGutamTsukaMeandim})
for all free, minimal $\mathbb{Z}^d$-actions on compact Hausdorff spaces $X$, regardless of the dimension of $X$. 
As the UCT is preserved through crossed products of $\mathbb{Z}^d$-actions \cite{RS}, Corollary~\autoref{corintro:uRp} yields the following corollary.

\begin{corintro} \label{cor:Z}
Let $A$ be a simple,
separable, unital, nuclear, $\mathcal{Z}$-stable
\cstar-algebra such that $T(A)$ is a non-empty Bauer simplex, and 
let $\alpha\colon \mathbb{Z}^d\to \mathrm{Aut}(A)$ be an action such
that the induced action $\mathbb{Z}^d \curvearrowright \partial_e T(A)$ is free and minimal. Then $A \rtimes_\alpha \mathbb{Z}^d$ is a simple,
separable, unital, nuclear, $\mathcal{Z}$-stable
$C^*$-algebra. Moreover, if $A$ is classifiable, then so is $A \rtimes_\alpha \mathbb{Z}^d$.
\end{corintro}

\autoref{s.URP} is closed with a further class of examples to which our methods apply. More precisely, it is not hard to check
that all dynamical systems as in Question~\ref{qstA} such that the set of invariant traces is a finite-dimensional simplex,
verify the McDuff property with respect to invariant traces.
As a consequence, Theorem \ref{thmintro.B} entails the following corollary.
\begin{corintro}[{\autoref{thm:main2.}}] \label{cor:mono}
	Let $G$ be a countable, discrete, amenable group, let $A$ be a simple, separable, unital, nuclear $\mathcal{Z}$-stable
	\cstar-algebra, and let $\alpha \colon G \to\mathrm{Aut}(A)$ be an outer action such that the set of invariant traces is a finite-dimensional simplex. Then $A \rtimes_\alpha G$ is a simple,
separable, unital, nuclear, $\mathcal{Z}$-stable $C^*$-algebra.
\end{corintro}

The paper is structured as follows. In \autoref{S.prel} we cover some preliminary notions and results needed in later sections.
In \autoref{S.comparison} we introduce the definition of dynamical subequivalence (\autoref{df:DynSubEq}), weak dynamical comparison (\autoref{df:wdc}) and we prove Theorem~\autoref{thmB}. In \autoref{S.4} we define the McDuff property with respect to invariant traces (\autoref{def:wuMc}) and we
prove Theorem~\autoref{thmintro.B}. In \autoref{S.rokhlin} we present the definition of weak Rokhlin property with respect to invariant traces for actions on $C^*$-algebras (\autoref{df:iwRp}), and we show that it implies McDuffness with respect to invariant traces of the action, if the starting $C^*$-algebra is uniformly McDuff (\autoref{lemma:cpcoz}). Finally, \autoref{s.URP} is devoted to the proof of Corollaries~\autoref{corintro:uRp}
and \ref{cor:mono}.

After uploading the preprint to the arxiv, we learned that 
Li, Niu and Wang have independently obtained 
results similar to some of ours in an ongoing project.

\vspace{0.2cm}

\textbf{Acknowledgements:} The authors would like to thank 
Stuart White for helpful conversations. Most of this work was done
while the first three named authors were visiting the fourth one. 
We thank the Institut de Math\'ematiques 
de Jussieu-Paris Rive Gauche for the hospitality.

\section{Preliminaries} \label{S.prel}

\subsection{Functional calculus for order zero maps}
We briefly cover some basics on completely positive
order zero maps, and we refer the reader to 
\cite{WinZac_cpcoz} for more details.

\begin{df} 
Let $A$ and $B$ be $C^*$-algebras, and let $\varphi\colon A\to B$ be a completely positive map. We say
that $\varphi$ has \emph{order zero} if for every $a,b\in A$
with $ab=0$, we have $\varphi(a)\varphi(b)=0$.\end{df}

By \cite[Theorem 2.3 and Corollary 3.1]{WinZac_cpcoz}, 
there is a natural one-to-one correspondence 
between completely positive 
contractive order zero maps $A\to B$
and homomorphisms 
$C_0((0,1])\otimes A\to B$. 
Under this correspondence, 
a completely positive contractive order zero map
$\varphi\colon A\to B$ induces the homomorphism 
$\rho_\varphi\colon C_0((0,1])\otimes A\to B$ determined by
$\rho_\varphi(\mathrm{id}_{(0,1]}\otimes a)=\varphi(a)$ for all $a\in A$. Conversely, if $\rho\colon C_0((0,1])\otimes
A\to B$ is a homomorphism, then the induced completely positive contractive order zero map $\varphi_\rho\colon A\to B$
is given by $\varphi_\rho(a)=\rho(\mathrm{id}_{(0,1]}\otimes a)$ for all $a\in A$.
This assignment allows to consider functional calculus for 
completely positive maps of order zero.

\begin{prop}[{\cite[Corollary 3.2]{WinZac_cpcoz}}] \label{df:FuncCalcOz}
Let $A$ and $B$ be $C^*$-algebras, let $\varphi\colon A\to B$ be a completely positive contractive order zero map, let 
$\rho_\varphi\colon C_0((0,1])\otimes A\to B$ be the 
associated homomorphism as above, and let 
$f\in C_0((0,1])_+$. Then the function $f(\varphi)\colon A\to B$ defined as $f(\varphi)(a)=\rho_\varphi(f\otimes a)$ for 
all $a\in A$, is a completely positive contractive order zero map.
\end{prop}

We will need the following observation.

\begin{lma}\label{prop:FunctCalcOzMaps}
Let $A,B$ and $C$ be $C^*$-algebras, 
let $\varphi\colon A\to B$ 
be a completely positive contractive order zero map, let
$q\colon B\to C$ be a homomorphism, and
let $f\in C_0((0,1])$ be a positive contraction. Then 
$ q\circ f(\varphi)=f(q\circ\varphi)$.
\end{lma}
\begin{proof}
Adopt the notation from the discussion above. One readily checks that 
$q\circ \rho_\varphi=\rho_{q\circ\varphi}$, by verifying the equality
on elements of $\mathrm{id}_{(0,1]}\otimes A$, which 
generate $C_0((0,1])\otimes A$. Using this at the second 
step, for $a\in A$ we conclude that 
\[(q\circ f(\varphi))(a)=q(\rho_\varphi(f\otimes a))=
\rho_{q\circ\varphi}(f\otimes a)=f(q\circ\varphi)(a).\qedhere
\]
\end{proof}

\subsection{Norm and tracial ultrapowers}

Given a $C^*$-algebra $A$, we write $\ell^\infty(A)$ for the $C^*$-algebra
of all bounded sequences in $A$ with the supremum norm and pointwise
operations.

\begin{df}\label{def:normUltPwr}
	Let $A$ be a $C^*$-algebra and $\mathcal{U}$ be a free ultrafilter on $\mathbb{N}$. We write $c_\mathcal{U}$ for the ideal of $\ell^\infty(A)$ given by 
	\[c
_\mathcal{U}(A)=\big\{(a_n)_{n\in\mathbb{N}}\in \ell^{\infty}(A)\colon \lim_{n\to\cU}\|a_n\|=0 \big\}.\]
The \emph{norm ultrapower} of $A$ is defined to be the quotient
$A_\mathcal{U}= \ell^{\infty}(A)/c_\mathcal{U}(A)$. 
\end{df}

Note that $A$ embeds canonically into $\ell^\infty(A)$ and into 
$A_\mathcal{U}$ via constant
sequences. We write $A_\U\cap A'$ for the relative commutant of
$A$ in $A_\U$; this algebra is usually called the \emph{central sequence
algebra} of $A$. 

Given a $C^*$-algebra $A$, we let $T(A)$ be the set of its tracial states. Thorough this paper, we shall refer to tracial states
simply as \emph{traces}.
For a trace $\tau$ on a unital $C^*$-algebra $A$, the associated \emph{trace seminorm} on $A$ is given by
$\|a\|_{2,\tau}=\tau(a^*a)^{1/2}$
for all $a\in A$.  For a non-empty closed subset $T\subseteq T(A)$, we set
$\|\cdot\|_{2,T}= \sup_{\tau\in T}\|\cdot\|_{2,\tau}$.

\begin{df}\label{def:TrUltPwr}
Let $A$ be a $C^*$-algebra, fix  a free ultrafilter $\mathcal{U}$ on $\mathbb{N}$ and a closed non-empty subset $ T\subseteq T(A)$.  We write $J_T$ for the ideal of $A_\cU$ given by
\[J_T= \big\{(a_n)_{n\in\mathbb{N}}\in A_\cU\colon \lim_{n\to \cU}\|a_n\|_{2,T}=0 \big\}. \]
The \emph{uniform tracial ultrapower} of $A$ is defined to be the quotient
$A^\mathcal{U}= A_\cU/J_{T(A)}$. We denote by 
$\kappa\colon A_\cU \to A^{\cU}$ the canonical quotient map.
\end{df}

If $A$ is a unital $C^*$-algebra and $T\subseteq T(A)$ is a closed non-empty subset, then $\|\cdot\|_{2,T}$ is a seminorm. If $T$ contains a faithful trace, or more generally if $T$ separates positive elements, then $\|\cdot\|_{2,T}$ is a norm on $A$. This is always true for instance when $A$ is simple with non-empty trace space. Thorough this paper we will often consider the case of a discrete, amenable group acting on a $C^*$-algebra $A$, and $T = T(A)^\alpha$, that is, the set of all invariant traces. Note that the latter is then non-empty by Day's Fixed Point Theorem
(\cite[Theorem 1.3.1]{runde}).

Regardless of the context, we will always consider situations where $\|\cdot\|_{2,T(A)}$ and $\|\cdot\|_{2,T(A)^\alpha}$ are norms on $A$, even when we do not explicitly say so.
In this case the canonical embedding $A\hookrightarrow A_\U$ induces an injective map 
$A\to A^\U$. We write $A^\U\cap A'$ for the relative commutant of the image
of $A$ inside $A^\U$.

Given any $\tau \in T(A)$, we can naturally extend it to a trace on $A_\cU$ (and $A^\cU$) by setting
\[
\tau((a_n)_{n \in \mathbb{N}}) = \lim_{n \to \cU} \tau(a_n), \text{ for every } (a_n)_{n \in \mathbb{N}} \in A_\cU \text{ (or in $A^\cU$)}.
\]
Keeping this in mind, with a slight abuse of notation we denote such extension again by $\tau$, and throughout the paper we will identify $T(A)$
with a subset of $T(A_\cU)$ and $T(A^\cU)$.

If $\alpha\colon G\to\mathrm{Aut}(A)$ is an action of a 
discrete group $G$ on
a $C^*$-algebra $A$, then $\alpha$ naturally acts coordinatewise on $A_\cU$
and $A^\cU$. We denote such actions by $\alpha_\cU$ and $\alpha^\cU$ respectively. Moreover, the quotient map $\kappa$ is equivariant with respect
to these actions.

In the absence of group actions, the following proposition is
\cite[Proposition~4.5 and Proposition~4.6]{KirchRor_centr}.
The last claim is proved in
\cite[Lemma~7.5]{GHV}.

\begin{prop}\label{lemma:surjective}
Let $A$ be a simple, separable, unital $C^*$-algebra, let 
$G$ be a countable, discrete group, and let 
$\alpha\colon G \to \mathrm{Aut}(A)$ be an action.
Then the quotient map $\kappa\colon A_\cU \to A^\cU$ restricts to a surjective equivariant map
\[
\kappa\colon \big(A_\cU\cap A',\alpha_\cU\big)\to \big(A^\cU\cap A',
\alpha^\cU\big).
\]
Moreover, $\kappa$ restricts to a quotient map $\left(A_\cU\cap A'\right)^{\alpha_\cU}\to \left(A^\cU\cap A'\right)^{\alpha^\cU}$.
\end{prop}


\section{Dynamical comparison for actions of amenable groups} \label{S.comparison}

%

%

The main goals of this section are introducing the notion of \emph{weak dynamical comparison} (\autoref{df:wdc})
and proving \autoref{thm:wComp}, of which Theorem \ref{thmB} in the introduction is a direct consequence.
Before doing so, we recall some elementary notions from the 
order theory of $C^*$-algebras, in particular the Cuntz
semigroup. 

\begin{df}\label{df:CuntzEq}
	Let $A$ be a $C^*$-algebra. Let $a,b\in A_+$. \begin{enumerate}
		\item We say that 
			\textit{$a$ is Cuntz subequivalent to $b$ (in $A$)} and write $a\precsim b$, if there exists a sequence $(r_n)_{n\in\mathbb{N}}$ in $A$ such that $\lim_{n\to\infty}\|r_nbr_n^*-a\|=0$.
		\item We say that \textit{$a$ is Cuntz equivalent to $b$} and write $a\sim b$, if $a\precsim b$ and $b\precsim a$.
	\end{enumerate}

The \textit{Cuntz semigroup} of $A$ is defined as 
	\[\mathrm{Cu}(A)= (A\otimes \mathcal{K})_+/\sim. \]
	
	The class of an element $a\in (A\otimes \mathcal{K})_+$ in $\mathrm{Cu}(A)$ will be denoted by $[a]$. The set $\mathrm{Cu}(A)$ is an ordered abelian semigroup (in fact, a monoid) with respect to addition induced by direct sums and order induced by the Cuntz subequivalence.	
\end{df}

For $\delta>0$ and a self-adjoint element $a\in A$, we let $(a-\delta)_+$ denote the function $[0,\infty)\to \mathbb{R}$ given by $x\mapsto \max\{x-\delta,0\}$ applied on $a$ using functional calculus.

For $n\in\N$, let 
$f_n\colon [0,\infty)\to [0,1]$ be the continuous function which 
vanishes on $[0,\frac{1}{n}]$, is equal to 1 on $[\frac{2}{n},\infty)$,
and is linear on $[\frac{1}{n},\frac{2}{n}]$.

\begin{df}\label{df:dtau}
Let $A$ be a unital $C^*$-algebra and let $QT(A)$ be the set of normalized 2-quasitraces on $A$ (see \cite[Definition 6.7]{Hannes}
for a definition in the unbounded case). Fix $\tau\in QT(A)$, and 
let $a\in (A\otimes\K)_+$. The quasitrace $\tau$ uniquely extends to an unbounded quasitrace on $A \otimes \K$, which we still denote $\tau$. We set 
$d_\tau(a)= \lim_{n\to\infty} \tau(f_n(a))\in [0,\infty]$.
Equivalently, we have
\[d_\tau(a)=\sup_{\ep>0} \lim_{n\to\I} \tau((a-\ep)_+^{1/n}).\]

Given $a\in (A\otimes \mathcal{K})_+$, the \textit{rank function of $a$}, denoted 
$\mathrm{rk}(a)\colon QT(A)\to [0,\infty]$, is the map defined as
	\[\mathrm{rk}(a)(\tau)= d_\tau(a)\]
for all $\tau\in QT(A)$, which are identified with the corresponding extensions on $A \otimes \K$.
\end{df}

The rank function $\mathrm{rk}(a)$ is a lower semicontinuous function on $QT(A)$. If moreover
$A$ is simple and $a\neq 0$, then $\mathrm{rk}(a)$ is strictly positive.

It is easy to see that $a \precsim b$ implies $\rk(a) \le \rk(b)$.
In particular, 
$a \sim b$ implies $\rk(a) = \rk(b)$. It follows that 
the rank function is well-defined on the Cuntz semigroup.

\begin{df}\label{df:StrictComp}
Let $A$ be a unital $C^*$-algebra. We say that $A$ has \textit{strict comparison} if whenever $a,b\in (A\otimes \mathcal{K})_+$ with $b\neq 0$ satisfy $\mathrm{rk}(a)<\mathrm{rk}(b)$, then $a\precsim b$.
\end{df}

In this paper we will exclusively be interested in unital $C^*$-algebras for which all quasitraces are traces (this is the case
for instance for exact $C^*$-algebras \cite{HaaQT}). We will therefore systematically consider the rank functions as maps defined on $T(A)$,
and read strict comparison as a property about comparison on traces.

Given a partially ordered set $(S,\leq)$ and $s,t\in S$, it is said that $s$ is \emph{(sequentially) compactly contained in} $t$, in symbols $s\ll t$, if for all increasing sequences $(t_n)_{n\in\mathbb{N}}$ such that $\sup_{n\in\mathbb{N}} t_n$ exists and $t\leq \sup_{n\in\mathbb{N}} t_n$, we have that $s\leq t_{n}$ for some $n\in\mathbb{N}$.

\begin{lma}\label{ContFuncBetween}
	Let $A$ be a simple, separable, unital $C^*$-algebra such that $QT(A) = T(A)$, let $a \in (A\otimes \mathcal{K})_+\setminus\{0\}$, let $\ep > 0$ and let $\gamma\in (0,1)$. Then there exists a continuous affine function $f\colon T(A) \to [0,\infty]$ such that
	\[(1-\gamma)\cdot \mathrm{rk}\left((a-\ep)_+\right) < f < \mathrm{rk}(a).\]
\end{lma}

\begin{proof}
Since $\mathrm{Cu}(A)\cong \mathrm{Cu}(A\otimes\mathcal{K})$, the combination of
\cite[Corollary~2.56]{Hannes} and \cite[Lemma~3.14]{Hannes} yields that $[(a-\ep)_+]\ll [a]$ in $\mathrm{Cu}(A)$. By \cite[Lemma~2.2.5]{Leonel13}, it follows that 

\begin{equation}\label{eq:rkineq}
\left(1-\gamma/2\right)\cdot\mathrm{rk}\left((a-\ep)_+\right)\ll \left(1-\gamma/4\right)\cdot\mathrm{rk}(a).
\end{equation}
	
Since $A$ is simple and $a\neq 0$, we have that $\mathrm{rk}(a)\colon T(A)\to [0,\infty]$ is a strictly positive function. Moreover, it is well-known that $\mathrm{rk}(a)\colon T(A)\to [0,\infty]$, as a lower semicontinuous affine function, is a pointwise limit of an increasing net of continuous affine functions (see \cite[Corollary~I.1.4]{Alfsen71}). 
	As $T(A)$ is a compact and metrizable space, we can find an increasing sequence $(f_n)_{n\in\mathbb{N}}$ of strictly positive continuous affine functions $f_n\colon T(A)\to [0,\infty]$ such that $\sup_{n\in\mathbb{N}}f_n(\tau)=\left(1-\gamma/4\right)\cdot\mathrm{rk}(a)(\tau)$ for all $\tau\in T(A)$ (see \cite[Lemma~4.2]{TikuisisToms}).
	
	By \eqref{eq:rkineq}, there exists $n\in\mathbb{N}$ such that 
	\[\left(1-\gamma/2\right)\cdot \mathrm{rk}\left((a-\ep)_+\right)\leq f_n \leq \left(1-\gamma/4\right)\cdot\mathrm{rk}(a)<\mathrm{rk}(a). \]
	
	Set $f = f_n$. Simplicity of $A$ implies that $\mathrm{rk}\left((a-\ep)_+\right)$ is either strictly positive or constantly zero. The latter happens exactly when $\left(a-\ep\right)_+=0$, and then we are done since $f_n$ was chosen to be strictly positive. On the
	other hand, if $\left(a-\ep\right)_+\neq 0$, we have 
	
	\[\left(1-\gamma\right)\cdot \mathrm{rk}\left((a-\ep)_+\right)<\left(1-\gamma/2\right)\cdot \mathrm{rk}\left((a-\ep)_+\right)\leq f <\mathrm{rk}(a),
	\]
as desired.
\end{proof}

We turn to the definition of noncommutative dynamical subequivalence, an analogue of dynamical subequivalence (see \cite[Definition 3.1]{Ker_dimension_2020})
originating from an upcoming work by Bosa, Perera, Wu and Zacharias \cite{BPWZ_dynamical_2022}.

\begin{df}\label{df:DynSubEq}
Let $\alpha\colon G\to \mathrm{Aut}(A)$ be an action of a countable, discrete group $G$ on a $C^*$-algebra $A$. 
Let $a,b\in A_+$.  
	
\begin{enumerate}
\item We say that $a$ is \textit{elementarily dynamically subequivalent} to $b$, and write $a\precsim_0 b$, if for any $\ep>0$ there are $\delta>0$, $n\in\mathbb{N}$, elements $g_1,\ldots,g_n\in G$, and $x_1,\ldots, x_n\in (A\otimes \mathcal{K})_+$ such that 
\[(a-\ep)_+\precsim \oplus_{j=1}^{n} \alpha_{g_j}(x_j) \quad \text{ and } \quad \oplus_{j=1}^{n}x_j\precsim (b-\delta)_+.\] 
	
\item \label{dse:item2}
We say that $a$ is \textit{dynamically subequivalent} to $b$ and write $a\precsim_\alpha b$, if there exist $m\in\mathbb{N}$ and elements $y_1,\ldots,y_m\in (A\otimes \mathcal{K})_+$ with
\[a=y_1\precsim_0 y_2\precsim_0 \ldots \precsim_0 y_m=b. \] 
\end{enumerate}
\end{df}

We currently do not know whether the relation $\precsim_0$ is transitive, hence the need for item \eqref{dse:item2} in the definition of
$\precsim_\alpha$.
The following easy observations are isolated for later use.

\begin{rem}\label{rem: CompNoepsDel}
	Let $a,b\in A_+$. Suppose that there are $n\in\mathbb{N}$, elements $g_1,\ldots,g_n\in G$, and $x_1,\ldots, x_n\in (A\otimes \mathcal{K})_+$ such that 
	\[a\precsim \oplus_{j=1}^{n} \alpha_{g_j}(x_j) \quad \text{ and } \quad \oplus_{j=1}^{n}x_j\precsim b.\] 
It then follows from a double application of \cite[Proposition~2.4]{Ror_1992} that $a\precsim_0 b$. \end{rem}

\begin{rem}\label{rem: DyncomVScom}
Let $\alpha\colon G\to \mathrm{Aut}(A)$ be an action of a countable, discrete group $G$ on a $C^*$-algebra $A$, and
let $a,b\in A_+$. It is easy to see that if $a\precsim_\alpha b$ then $a\precsim b$ in $A\rtimes_r G$, using the fact that $\alpha$ is unitarily implemented inside $A\rtimes_r G$.
\end{rem}

Using the notion of dynamical subequivalence, and inspired
by the concept of strict comparison, 
Bosa, Perera, Wu and 
Zacharias defined \emph{dynamical strict comparison} for 
a dynamical system; see \cite{BPWZ_dynamical_2022}. 
In the present work, we will need 
the following weakening of this notion, which is motivated 
by \cite[Definition~2.1]{Nar_polynomial_2021}. For our convenience, we state it in terms of traces (instead of quasitraces).

\begin{df} \label{df:wdc}
Let $A$ be a unital $C^*$-algebra, let $G$ be a discrete group, and let $\alpha \colon G \to \mathrm{Aut}(A)$ be an action. 
We say that $\alpha$ has \emph{weak dynamical comparison} if whenever $a,b \in (A\otimes \mathcal{K})_+$ with 
	$b\neq 0$ satisfy
	\[
	\sup_{\tau \in T(A)^\alpha} d_\tau(a) < \inf_{\tau \in T(A)^\alpha} d_\tau(b),\]
	then $a\precsim_\alpha b$.
\end{df}

\begin{df}
Let $A$ be a simple, unital \cstar-algebra. We say that 
\emph{all ranks are realized} if for 
every lower-semicontinuous, affine function $QT(A)\to (0,\infty]$
there exists $a\in (A\otimes\mathcal{K})_+$
such that $\mathrm{rk}(a)=f$.
\end{df}

We point out that it is conjectured that all ranks are realized for any nonelementary, simple, separable, stably finite, unital $C^*$-algebra. (And there is evidence that this may also hold 
outside of the simple case; see \cite{APRT}.)
Examples of simple \cstar-algebras where all ranks are realized
include those \cstar-algebras of stable rank one, by
\cite[Theorem~8.11]{ThielSR1}. In particular, this holds for
simple, stably finite, unital, $\mathcal{Z}$-stable $C^*$-algebras by \cite[Theorem~6.7]{Ror_2004}.

\begin{thm}\label{thm:wComp}
Let $G$ be a countable, discrete, amenable group and let $A$ be a simple, separable, unital \cstar-algebra with strict comparison,
such that $QT(A) = T(A)$ and such that all ranks are realized.
Then every action of $G$ on $A$ has weak dynamical comparison.
\end{thm}

\begin{proof}
Let $\alpha\colon G\to\mathrm{Aut}(A)$ be an action.
Let $a, b \in (A\otimes \mathcal{K})_+$ with $b\neq 0$ 
and such that 
\[\sup_{\tau \in T(A)^\alpha} d_\tau(a) < \inf_{\tau \in T(A)^\alpha} d_\tau(b).\]
Find constants $c_1, c_2$ and $c$ satisfying
\begin{equation}\label{eq:c1c2}
\sup_{\tau \in T(A)^\alpha} d_\tau(a) < c_1<c<c_2< \inf_{\tau \in T(A)^\alpha} d_\tau(b).
\end{equation}
Since all ranks are realized, 
there exists $y \in (A\otimes \mathcal{K})_+$ such that
$\mathrm{rk}(y) \equiv c$. It is sufficient to show that $a \precsim_0 y$ and $y \precsim_0 b$. 
	
We first show that $a \precsim_0 y$. Let $\ep > 0$ be given. By 
\autoref{ContFuncBetween}, there exists a continuous affine function $f_0\colon T(A)\to [0,\infty]$ such that
	\[\frac{c_1}{c}\cdot \mathrm{rk}\left((a-\ep)_+\right)<f_0<\mathrm{rk}(a). \]
	Set $f = \frac{c}{c_1}\cdot f_0$. Then
	\begin{equation}\label{eq:Cutaf}
		\mathrm{rk}\left((a-\ep)_+\right) < f
	\end{equation}
	and 
	\begin{equation}\label{eq:fOnInv}
		\max_{\tau\in T(A)^\alpha} f(\tau) < c. 
	\end{equation}
	
	Let $(F_n)_{n\in\mathbb{N}}$ be a F{\o}lner sequence in $G$. We claim that for some $n\in\mathbb{N}$ we have
	\begin{equation} \label{eq:ftau}
	\max_{\tau\in T(A)}	f\Big(\frac{1}{\left|F_n\right|}\sum_{g \in F_n}g\cdot\tau\Big) < c.
\end{equation}
Indeed, suppose that for each $n\in\mathbb{N}$ there exists a trace $\tau_n$ for which the above inequality fails. By compactness of $T(A)$, the sequence $\big(\frac{1}{\left|F_n\right|}\sum_{g \in F_n}g\cdot\tau\big)_{n\in\mathbb{N}}$ has a subsequence converging to some $\tau_0 \in T(A)$. This trace $\tau_0$ is easily seen to be invariant. Since $f$ is continuous, it follows that $f(\tau_0) \ge c$, which contradicts \eqref{eq:fOnInv}.
	
Note that $\tau\mapsto f\Big(\frac{1}{\left|F_n\right|}\sum_{g \in F_n}g\cdot\tau\Big)$ is a continuous function and $\big(\mathrm{rk}((y-\delta)_+)\big)_{\delta>0}$ is an increasing net of lower semicontinuous functions converging pointwise to 
$c$ as $\delta \to 0$. Thus, by the Dini-Cartan lemma (see \cite[Lemma~2.2.9]{Helms}), there exists $\delta>0$ such that
	\begin{equation}\label{eq:dini}
		f\Big(\frac{1}{\left|F_n\right|}\sum_{g \in F_n}g\cdot\tau\Big) < \mathrm{rk}\left(\left(y-\delta\right)_+\right)(\tau), \quad \text{ for all } \tau\in T(A).
	\end{equation}

Since $f$ is a strictly positive continuous, affine function, the same is true for $\frac{1}{\left|F_n\right|}f$. Since all ranks
are realized, there is $x\in (A\otimes \mathcal{K})_+$ such that $\mathrm{rk}(x) = \frac{1}{\left|F_n\right|}f$. Then
\begin{align*}
\mathrm{rk}\left(\left(a-\ep\right)_+\right) \overset{\mathclap{\eqref{eq:Cutaf}}}< f&  =\sum\limits_{g\in F_n}\frac{1}{\left|F_n\right|}f \\
& =\sum\limits_{g\in F_n}\mathrm{rk}(x) \\
& =\mathrm{rk}\Big(\oplus_{g\in F_n}x\Big).
\end{align*}
Moreover,
\begin{align*}
\mathrm{rk}\Big(\oplus_{g\in F_n}\alpha_{g^{-1}}(x)\Big) & =\sum_{g\in F_n}\mathrm{rk}\left(\alpha_{g^{-1}}(x)\right) \\
& =\sum\limits_{g\in F_n}\frac{1}{\left|F_n\right|}g^{-1}\cdot f \\
& \overset{\mathclap{\eqref{eq:dini}}}<\mathrm{rk}\left(\left(y-\delta\right)_+\right).
\end{align*}

Since $A$ has strict comparison, the computations above yield that
\begin{equation}\label{eq:cutaBelowX}
\left(a-\ep\right)_+\precsim \oplus_{g\in F_n}x \quad \text{ and } \quad \oplus_{g\in F_n} \alpha_{g^{-1}}(x)\precsim (y-\delta)_+.
\end{equation}
This completes the proof that $a\precsim_0 y$. 

We now turn to
the proof of $y\precsim_0 b$, which is similar. By \eqref{eq:c1c2}, we have 
\[c_2<\min_{\tau\in T(A)^\alpha}\mathrm{rk}(b)(\tau).\]
Again by the Dini-Cartan lemma, there is $\eta>0$ such that
	\begin{equation}\label{eq:c2etab}
c_2<\min_{\tau\in T(A)^\alpha}\mathrm{rk}\left(\left(b-\eta\right)_+\right)(\tau).\end{equation}
	By \autoref{ContFuncBetween}, there exists a continuous affine function $h_0\colon T(A)\to [0,\infty]$ such that
	\begin{equation}\label{eq:rkbeta}
	\frac{c}{c_2}\cdot \mathrm{rk}\left(\left(b-\eta\right)_+\right)<h_0<\mathrm{rk}(b). 
	\end{equation}
Set $h=\frac{c_2}{c}\cdot h_0$, and note that 
$\mathrm{rk}\left(\left(b-\eta\right)_+\right)<h$. 
Combining this with \eqref{eq:c2etab}, we get
\begin{equation*}
    c_2<\min_{\tau\in T(A)^\alpha} h(\tau).
\end{equation*}
Using compactness of $T(A)$ and continuity of $h$, analogously to what we did before for $f$ in \eqref{eq:ftau}, find $n\in\mathbb{N}$ such that 
\begin{equation}\label{eq:c2h}
		c_2<\min_{\tau\in T(A)}h\Big(\frac{1}{\left|F_n\right|}\sum\limits_{g\in F_n}g\cdot \tau\Big).
	\end{equation}
Since all ranks are realized, 
there exists $z\in (A\otimes \mathcal{K})_+$ such that $\mathrm{rk}(z)=\frac{1}{\left|F_n\right|}h_0$. We compute
\begin{align*}
\mathrm{rk}(y) & =\frac{c}{c_2}\cdot c_2 \\
& \overset{\mathclap{\eqref{eq:c2h}}}<  \frac{c}{c_2}\cdot  \sum\limits_{g\in F_n}\frac{1}{\left|F_n\right|}g^{-1}\cdot h \\
 & =\sum\limits_{g\in F_n}\frac{1}{\left|F_n\right|}g^{-1}\cdot h_0 \\
 & =\mathrm{rk}\Big(\oplus_{g\in F_n}\alpha_{g^{-1}}(z)\Big).
 \end{align*}
	Moreover,
	\begin{align*}
	\mathrm{rk}\Big(\oplus_{g\in F_n} z \Big) & =\sum\limits_{g\in F_n}\mathrm{rk}(z) \\
	& =\sum\limits_{g\in F_n}\frac{1}{\left|F_n\right|}h_0 \\
	& =h_0 \\
	&\overset{\mathclap{\eqref{eq:rkbeta}}} <\mathrm{rk}(b).
	\end{align*}
	These computations, together with strict comparison of $A$, yield that
	\[y\precsim \oplus_{g\in F_n}\alpha_{g^{-1}}(z)  \quad \text{ and } \quad \oplus_{g\in F_n}z\precsim b. \]
	We deduce that $y\precsim_0 b$ by \autoref{rem: CompNoepsDel}, 
	as desired. This finishes the proof.
\end{proof} 

Theorem \ref{thmB}, stated here as a corollary, is a particular case of \autoref{thm:wComp}. 

\begin{cor}\label{cor:thmB}
Let $G$ be a countable, discrete, amenable group and let $A$ be a simple, separable, unital, exact, $\mathcal{Z}$-stable \cstar-algebra.
Then every action of $G$ on $A$ has {weak dynamical comparison}.
\end{cor}
\begin{proof}
Note that simple, unital, exact, $\mathcal{Z}$-stable $C^*$-algebras verify $QT(A) = T(A)$ and have 
strict comparison by \cite{HaaQT} and 
\cite[Theorem~4.5]{Ror_2004}. Moreover, all ranks on them are realized by the combination
of \cite[Theorem~8.11]{ThielSR1} and \cite[Theorem~6.7]{Ror_2004}.
Thus the result follows from \autoref{thm:wComp}. 
\end{proof}

\section{McDuffness with respect to invariant traces and \texorpdfstring{$\mathcal{Z}$}{Z}-stability} \label{S.4}

We start this section by introducing the McDuff property with respect to invariant traces for actions
of discrete groups on separable, unital $C^*$-algebras. This is the main new definition of the present paper.
The rest of the section is devoted to the proof of \autoref{thm:main1}, of which Theorem \ref{thmintro.B} is a corollary.

\begin{df} \label{def:wuMc}
Let $A$ be a separable, unital $C^*$-algebra, 
let $G$ be a discrete group, and let $\alpha\colon G\to\Aut(A)$ be an
action with $T(A)^\alpha\neq\emptyset$. We say that $(A,\alpha)$
has the \emph{McDuff property with respect to invariant traces} if for every $d\in\N$ there exists
a unital homomorphism $M_d\to (A_\U\cap A')^{\alpha_\U}/J_{T(A)^\alpha}$.
\end{df}

Let $A$ be a unital $C^*$-algebra. An automorphism $\alpha\in \mathrm{Aut}(A)$ is called \textit{outer} if there does not exist a unitary $u\in A$ with $\alpha= \mathrm{Ad}(u)$. 
An action $\alpha\colon G\to \mathrm{Aut}(A)$ is called \emph{(pointwise) outer} if $\alpha_g$ is outer for every $g\in G\setminus\{e\}$.

The following is an analogue of \cite[Lemma~7.9]{PhiLarge} in the noncommutative setting. The result is probably known to some experts, but
we include a proof since we could not find it in the literature.
\begin{lma}\label{lma:Kishnoncom}
	Let $A$ be a simple, separable, unital \cstar-algebra, let $G$ be a countable, discrete group, let
	$\alpha \colon G \to \mathrm{Aut}(A)$ be an outer action, 
	and let $b\in A\rtimes_{r,\alpha} G$ be a positive non-zero element. Then 
	there exists a positive element $a \in A\setminus\{0\}$ such that
	$a \precsim b$ in $A\rtimes_{r,\alpha}G$.
\end{lma}
\begin{proof}
	Since $b\sim t b$ for any $t\in (0,\infty)$, 
	we can assume that $\|b\|= 1$.
	The canonical conditional expectation $E \colon A \rtimes_{r,\alpha} G \to A$ is faithful, and hence $E(b)$ is a non-zero positive element in $A$. 
	Fix $0 <\delta < \frac{\|E(b)\|}{8}$, and find a contraction $c \in A \rtimes_{\mathrm{alg}} G$ such that $\lVert c - b^{1/2} \rVert < \delta$. Write
	\[
	c^*c = E(c^*c)+\sum_{g \in K} c_g \lambda_g
	\]
	for some finite subset $K\subseteq G\setminus \{e\}$. Note that 
    $\|c^*c-b\|<2\delta$ and $\|cc^*-b\|<2\delta$. By R{\o}rdam's 
    lemma \cite[Proposition~2.4]{Ror_1992}, we get
	\begin{equation}\label{ccb}
(cc^*-2\delta)_+\precsim b
	\end{equation}
	Since $E$ is contractive, we also get
	\begin{equation}\label{eq:cb}
	\|E(c^*c)\|> \|E(b)\|-2\delta.
	\end{equation}

	Fix $0<\varepsilon<  \frac{\|E(b)\|}{2(|K|+1)}$. Apply \cite[Lemma 3.2]{KishOut}
	in order to obtain a positive element $h\in A$ with $\|h\|=1$ such that
	\begin{equation}\label{eq:esthEh}
	\|hE(c^*c)h\|\geq \|E(c^*c)\|-\varepsilon
	\end{equation}
	and
	\[\max_{g\in K}\|hc_g\alpha_g(h)\|\leq \varepsilon. \]
Since $\alpha_g=\mathrm{Ad}(\lambda_g)$ when restricted to the copy of $A$ inside $A\rtimes_{r,\alpha}G$ we have
		
	\[\max_{g\in K}\|hc_g\lambda_gh\|\leq \varepsilon. \]
	This, in particular, yields
	
	\begin{equation*}
	\|hc^*ch-hE(c^*c)h\|\leq |K|\varepsilon,
	\end{equation*}
which, in turn, implies that

	\begin{equation}\label{eq:hcE}
		(hE(c^*c)h-|K|\varepsilon-2\delta)_+\precsim (hc^*ch-2\delta)_+,
	\end{equation}
by \cite[Corollary~1.6]{PhiLarge}.
Set $a=(hE(c^*c)h-|K|\varepsilon-2\delta)_+$. 
Clearly $a\in A$ is a positive element. We claim that $a\neq 0$ and $a\precsim b$. To show that $a\neq 0$, we argue as follows:
	
	\begin{align*}
	\|a\|& \ \geq \ \|hE(c^*c)h\|-|K|\varepsilon-2\delta\\
	&\ \stackrel{\mathclap{{(\ref{eq:esthEh})}}}{\geq} \ \|E(c^*c)\|-(|K|+1)\varepsilon-2\delta\\
	& \ \stackrel{\mathclap{{(\ref{eq:cb})}}}{>} \
\|E(b)\|-(|K|+1)\varepsilon-4\delta\\
	&\ > \ \|E(b)\| -(|K|+1)\frac{\|E(b)\|}{2(|K|+1)}-4\frac{\|E(b)\|}{8}
	=0.
	\end{align*}
Thus $a\neq 0$. It remains to show that $a \precsim b$. We use  
\cite[Lemma~1.4(6)]{PhiLarge} at the second step and \cite[Lemma~1.7]{PhiLarge} at the third step, to get
	\[
	a \stackrel{\mathclap{{(\ref{eq:hcE})}}}\precsim (hc^*ch-2\delta)_+\sim (ch^2c^*-2\delta)_+\precsim (cc^*-2\delta)_+\stackrel{\mathclap{{(\ref{ccb})}}}{\precsim} b.\qedhere\]
\end{proof}



In \autoref{thm:main1} below, we will show that, under mild assumptions, a dynamical system $(A, \alpha)$ satisfying the
McDuff property with respect to invariant traces, produces a $\mathcal{Z}$-stable crossed product. In order to motivate the argument,
assume we are given a 
completely positive contractive order zero 
map 
\[\varphi \colon M_d \to (A_\cU \cap A')^{\alpha_\cU} \ \mbox{ with } \ 1 - \varphi(1) \in J_{T(A)^\alpha}\]
coming from McDuffness with respect to invariant traces.
Lifting along the canonical quotient maps $\ell^\infty(A)\to A_\cU \to A^\cU$ and using projectivity of the cone of $M_d$, we get a sequence $(\varphi_n)_{n\in\N}$ of order zero maps $\varphi_n\colon M_d\to A$ which, among other things, 
satisfy 
$\lim_{n\to\infty}\|1-\varphi_n(1)\|_{2,T(A)^\alpha}= 0$. In other
words, $1-\varphi_n(1)$ is small in all invariant traces. In order to conclude that the crossed product is $\mathcal{Z}$-stable,
we resort to the notion of tracial $\mathcal{Z}$-stability introduced in \cite{HO} (see \autoref{def:trZst}).
Because of this, what we need to show is that $1-\varphi_n(1)$ is small in the sense of Cuntz subequivalence. To do so, we need
to prove that $1-\varphi_n(1)$ is small in all 
\emph{dimension functions} associated with invariant traces, and then apply weak dynamical comparison (\autoref{thm:wComp}). In the following
proposition, we use functional calculus for order zero maps
to show that the lifts can be chosen to have this property.

\begin{prop}\label{prop:avgord0}	
Let $\alpha\colon G\to\mathrm{Aut}(A)$ be an action of a countable, discrete, amenable group $G$ on a separable, unital \cstar-algebra $A$ with non-empty trace space such
that $(A,\alpha)$ has the McDuff property with respect to invariant traces.
Let $d\in\mathbb{N}$,
let $F\subseteq A$ and $K\subseteq G$ 
be finite subsets, and let $\ep>0$.
Then there exists a completely positive contractive order zero
map $\rho\colon M_d\to A$ satisfying
	\begin{enumerate}
	\item $\|\rho(x)a-a\rho(x)\|<\ep\|x\|$ for all $x\in M_d$ and all $a\in F$,
	\item $\|\alpha_k(\rho(x))-\rho(x)\|<\ep\|x\|$ for all $x\in M_d$ and all $k\in K$,
	\item $d_\tau(1-\rho(1))<\ep$ for all $\tau\in T(A)^\alpha$.
	\end{enumerate}
\end{prop}
\begin{proof}
Use McDuffness with respect to invariant traces of $(A,\alpha)$ in order to choose
a completely positive contractive order zero map 
\[\varphi \colon M_d \to (A_\cU \cap A')^{\alpha_\cU} \ \mbox{ with } \ 1 - \varphi(1) \in J_{T(A)^\alpha}.\] 
By \cite[Proposition 1.2.4]{Win_covii}, there exists a completely positive contractive order zero map $\psi=(\psi_n)_{n \in \mathbb{N}}\colon M_d\to \ell^{\infty}(A)$ making the following diagram commute
\begin{equation*} 
\xymatrix{
 && \ell^\infty(A) \ar[d]^{q} \\%
M_d \ar[rr]_-{\varphi} \ar[urr]^{\psi} && A_\cU.
}
\end{equation*}
	
Let $f \in C_0\left((0,1]\right)$ be given by
\[f(t)\coloneqq \begin{cases}
2t, & \text{ if } t\in (0,1/2],\\
1, & \text{ if } t\in [1/2,1].
\end{cases}
\]
By \autoref{prop:FunctCalcOzMaps}, the following diagram commutes:	
\begin{equation*}
\xymatrix{
 && \ell^\infty(A) \ar[d]^{q} \\%
M_d \ar[rr]_-{f(\varphi)} \ar[urr]^{f(\psi)} && A_\cU
}
\end{equation*}
Moreover, $f(\psi)=\left(f(\psi_n)\right)_{n\in\mathbb{N}}$ and $f(\psi_n)(1)=f(\psi_n(1))$, for every $n\in\mathbb{N}$.
	
Since the range of $f(\varphi)$ is contained in $(A_\cU \cap A')^{\alpha_\cU}$ and
$1 - \varphi(1) \in J_{T(A)^\alpha}$, there exists $n \in \N$ such that
\begin{enumerate}[label=(\roman*)]
	\item $\| f(\psi_n)(x) a - af(\psi_n)(x) \| < \ep \| x \|$ for all $x \in M_d$ and all $a \in F$,
	\item $\| \alpha_k(f(\psi_n)(x)) - f(\psi_n)(x) \| < \ep \| x \|$ for all $x \in M_d$ and all $k \in K$,
	\item $\| 1 - \psi_n(1) \|_{2, T(A)^\alpha}^2 < \ep/4$.
	\end{enumerate}

Set $\rho= f(\psi_n)$. Then $\rho$ satisfies conditions (1) 
and (2) in the statement, and it remains to show that $d_\tau\left(1-\rho(1)\right)<\ep$ for all $\tau\in T(A)^\alpha$. Fix $\tau\in T(A)^\alpha$ and let $\mu_\tau$ be the Borel probability measure on the spectrum $\sigma\left(\psi_n(1)\right)$ representing the trace $\tau|_{C^*\left(1,\psi_n(1)\right)}$. 
Since $\sigma(\psi_n(1))\subseteq [0,1]$, we can regard $\mu_\tau$ as 
a probability measure on $[0,1]$ supported on $\sigma(\psi_n(1))$.
Then the following computation finishes the proof:
	\begin{align*}
	d_\tau(1-\rho(1))&=d_\tau(1-f(\psi_n)(1))
	=d_\tau(1-f(\psi_n(1)))\\
	&=\mu_\tau(\left\{x\in \sigma(\psi_n(1))\colon f(x)\neq 1 \right\})
	= 4  \cdot \frac{\mu_\tau(\left[0,\frac{1}{2}\right])}{4} \\
	& = 4  \int_0^{\frac{1}{2}}\Big(1 - \frac{1}{2}\Big)^2\, d\mu_\tau 
	\le 4  \int_0^{\frac{1}{2}} (1- t)^2\, d \mu_\tau \\
	&\le 4   \int_0^{1} (1- t)^2\, d \mu_\tau 
	 = 4  \| 1- \psi_n(1) \|^2_{2, \tau} 
	 < \ep.
	\end{align*}
\end{proof}

As anticipated, our proof of $\mathcal{Z}$-stability for crossed products
factors through Hirshberg-Orovitz' notion of tracial
$\mathcal{Z}$-stability, which we recall below.

\begin{df}[{\cite[Definition 2.1]{HO}}] \label{def:trZst}
An infinite-dimensional, unital \cstar-algebra $A$ is said to be \emph{tracially $\mathcal{Z}$-stable} if for every $d \in \mathbb{N}$,
every finite subset $F \subseteq A$, every $\ep > 0$ and every $a \in A_+\setminus\{0\}$, 
there exists a completely positive contractive order zero map $\varphi \colon M_d \to A$ such that
\begin{enumerate}
\item $1 - \varphi(1) \precsim a$,
\item\label{item:Z2} $\| [\varphi(b),c] \| < \ep \| b \|$ for all $b \in M_d$ and $c \in F$.
\end{enumerate} 
\end{df}

By \cite[Proposition~2.2 and Theorem~4.1]{HO} a simple, 
separable, unital, nuclear \cstar-algebra is $\mathcal{Z}$-stable if and only if it is tracially $\mathcal{Z}$-stable.

\begin{thm}\label{thm:main1}
Let $G$ be a countable, discrete, amenable group, let $A$ be a simple,
separable, unital, 
\cstar-algebra with $QT(A) =T(A)\neq\emptyset$, which has strict comparison and 
for which all ranks are
realized. Let $\alpha \colon G \to \mathrm{Aut}(A)$ be an outer 
action 
such that $(A,\alpha)$ has the McDuff property with respect to invariant traces. 
Then $A \rtimes_\alpha G$ is tracially $\mathcal{Z}$-stable.
If $A$ is moreover nuclear, then $A\rtimes_\alpha G$ is
$\mathcal{Z}$-stable. 
\end{thm}

\begin{proof}
The assumptions of the theorem, together with outerness of
$\alpha$, imply that $A\rtimes_\alpha G$ is a simple, separable, unital $C^*$-algebra (see \cite[Theorem~3.1]{KishOut}), which is also nuclear when $A$ is nuclear. Thus, by \cite[Theorem~4.1]{HO}, it is sufficient to prove that 
$A \rtimes_\alpha G$ is tracially $\mathcal{Z}$-stable.

Let $d\in\mathbb{N}$, let $F\subseteq A\rtimes_\alpha G$ be a finite subset, let $\ep>0$, and let $b\in \left(A\rtimes_\alpha G\right)_+\setminus\{0\}$. 
Without loss of generality, we may assume that there 
exist finite subsets $F_0\subseteq A$ and $K\subseteq G$ such that
$F=F_0\cup \{u_k\}_{k\in K}$. By \autoref{lma:Kishnoncom}, there exists $a\in A_+\setminus\{0\}$ such that $a\precsim b$ in $A\rtimes_\alpha G$.

Note that $T(A)^\alpha\neq \emptyset$ because $G$ is amenable.
Set $\delta\coloneqq \inf_{\tau\in T(A)^{\alpha}} d_\tau(a)$.
Since the rank function $\mathrm{rk}(a)\colon T(A)\to [0,\infty]$ is lower semicontinuous, it attains its minimum on compact subsets. Simplicity of $A$ therefore implies that $\delta>0$. 
Set $\eta= \min\{\ep,\delta/2\}$.
    
Apply \autoref{prop:avgord0} to find a completely positive contractive order zero map $\rho\colon M_d\to A$ satisfying
		\begin{enumerate}
		\item $\|\rho(x)y-y\rho(x)\|<\eta\|x\|$ for all $x\in M_d$ and all $y\in F_0$,
		\item $\|\alpha_k(\rho(x))-\rho(x)\|<\eta\|x\|$ for all $x\in M_d$ and all $k\in K$,
		\item $d_\tau(1-\rho(1))<\eta$ for all $\tau\in T(A)^\alpha$.
	\end{enumerate}
	It follows from (3) and \autoref{thm:wComp} that $1-\rho(1)\precsim_\alpha a$,
	and hence $1-\rho(1)\precsim b$ in $A\rtimes_\alpha G$ by \autoref{rem: DyncomVScom}.
Since properties (1) and (2) imply that 
\[\|\rho(x)c-c\rho(x)\|<\ep\|x\|\] 
for all $x\in M_d$ and all $c\in F$,
we conclude that 
$A\rtimes_\alpha G$ is tracially $\mathcal{Z}$-stable, as desired. 	
\end{proof}

We isolate the following useful corollary.

\begin{cor}\label{thm:main2}
Let $G$ be a countable, discrete, amenable group, let $A$ be a simple, separable, stably finite, unital, nuclear, $\mathcal{Z}$-stable
\cstar-algebra, and 
let $\alpha \colon G \to \mathrm{Aut}(A)$ be an outer 
action such that $(A,\alpha)$ has the McDuff property with respect to invariant traces. 
Then $A \rtimes_\alpha G$ is a simple,
separable, unital, nuclear, $\mathcal{Z}$-stable
$C^*$-algebra. 
\end{cor}
Note that the same corollary holds when $A$ is purely infinite, using \autoref{lma:Kishnoncom}.

\section{The weak Rokhlin property with respect to invariant traces} \label{S.rokhlin}

In this section we show that dynamical systems $(A, \alpha)$ such that $A$ is uniformly McDuff
and such that $\alpha$ satisfies a suitable Rokhlin-type condition, have the McDuff property with respect to invariant traces. We begin by
recalling the definition of uniform McDuffness.

\begin{df}[{\cite[Definition 4.2]{cetw}}] \label{df:McDuff}
Let $A$ be a separable, unital $C^*$-algebra with non-empty trace space. We say that $A$ is \emph{uniform McDuff} if for every $d \in \mathbb{N}$
there exists a unital embedding $M_d \to A^\cU \cap A'$.
\end{df}

We now introduce the notion of \emph{weak Rokhlin property with respect to invariant traces}.
Before doing so, remember that for a countable, discrete group $G$, a finite subset $K \subseteq G$, and $\ep>0$, we say that a finite subset $S\subseteq G$ is \emph{$(K,\ep)$-invariant} if 
	\[\Big\lvert \bigcap_{k\in K\cup\{e\}} kS\Big \rvert\geq (1-\ep)\left|S \right| . \]

\begin{df}\label{df:iwRp}
Let $\alpha\colon G\to\Aut(A)$ be an action of
a discrete, amenable group $G$ on a unital $C^*$-algebra $A$. 
We say that $\alpha$ has the \emph{weak Rokhlin property with respect to invariant traces}
if for every finite subset $K\subseteq G$ and every $\ep>0$, there
exist $(K,\ep)$-invariant finite subsets $S_1,\ldots,S_m\subseteq G$
and pairwise orthogonal positive contractions $a_{g,i}\in A^\mathcal{U}\cap A'$,
for $i=1,\ldots,m$ and $g\in S_i$, such that
\begin{enumerate}
\item \label{item:wrp1} $\alpha_g^{\mathcal{U}}(a_{h,i})=a_{gh,i}$, for all $i=1,\ldots,m$, $g\in G$ and $h\in S_i$ for which $gh\in S_i$,
\item \label{item:wrp2} $\tau\big(1-\sum_{i=1}^m\sum_{g\in S_i}a_{g,i}\big)<\ep$, for all $\tau\in T(A)^\alpha$.
\end{enumerate}
\end{df}

The weak Rokhlin property with respect to invariant traces differs from the weak (tracial)
Rokhlin property from \cite{GHV} or \cite{HO}, in that the 
remainder is only assumed to be small with respect to 
\emph{invariant} traces, and not all traces. This 
is not a minor difference, and actions with the
weak Rokhlin property with respect to invariant traces are easier to construct (see \autoref{prop:uRpiwRp}). The results in this section, and in 
particular $\mathcal{Z}$-stability of the crossed products (see \autoref{cor:iwRp}),
are new even for actions which have the Rokhlin properties considered in \cite{GHV} and \cite{HO}.

Observe that actions with the weak Rokhlin property with respect to invariant traces are outer.

\begin{lma}\label{lma:iwRpOuter}
	Let $G$ be a discrete, amenable group,
	let $A$ be a unital $C^*$-algebra, and let 
	$\alpha\colon G\to\Aut(A)$ be an action with the 
	weak Rokhlin property with respect to invariant traces. Then $\alpha$ is outer.
\end{lma}
\begin{proof}
	Let $k\in G\setminus\{e\}$ and assume by contradiction that $\alpha_k=\mathrm{Ad}(u)$, for some unitary $u\in A$. Find $(\{k\},1/2)$-invariant finite subsets $S_1,\ldots,S_m\subseteq G$
	and pairwise orthogonal positive contractions $a_{g,i}\in A^\mathcal{U}\cap A'$,
	for $i=1,\ldots,m$ and $g\in S_i$, such that
	\begin{enumerate}
		\item $\alpha_g^{\mathcal{U}}(a_{h,i})=a_{gh,i}$, for all $i=1,\ldots,m$, $g\in G$ and $h\in S_i$ for which $gh\in S_i$,
		\item $\tau\big(1-\sum_{i=1}^m\sum_{g\in S_i}a_{g,i}\big)<\ep$, for all $\tau\in T(A)^\alpha$.
	\end{enumerate}
	Note that condition~(1) implies that if $a_{s,i}=0$ for some $s\in S_i$ and $i\in \{1,\ldots,m\}$, then $a_{g,i}=0$, for all $g\in S_i$. Condition~(2) therefore forces the existence of $i\in\{1,\ldots,m\}$ for which $a_{g,i}\neq 0$, for every $g\in S_i$. By $(\{k\},1/2)$-invariance of $S_i$, there exists $s\in S_i$ such that $ks\in S_i$.
	Since $\alpha_k=\mathrm{Ad}(u)$ and $a_{s,i}\in A^\cU \cap A'$, condition~(1) implies that
	$a_{s,i}=a_{ks,i}$.
	Pairwise orthogonality implies in turn that $a_{s,i}=0$, contradicting the choice of $i$.
	
\end{proof}

The following proposition is inspired by Kerr's proof of \cite[Theorem~9.4]{Ker_dimension_2020} and isolates the most technical part of the averaging arguments needed in the main result of this section,
\autoref{lemma:cpcoz}.

\begin{prop}\label{lma:WeddingCake}
Let $G$ be a discrete, amenable group, let $A$ be a unital $C^*$-algebra,
let $\alpha\colon G\to\Aut(A)$ be an action with the weak Rokhlin property with respect to invariant traces, let $\ep>0$ and let $K\subseteq G$ be a finite subset.
Then there exist finite subsets $T_1,\ldots,T_m\subseteq G$, and pairwise orthogonal positive contractions $a_{g,i}\in A^\mathcal{U}\cap A'$, for 
$i=1,\ldots,m$ and $g \in T_i$, satisfying the following conditions:
\begin{enumerate}
\item \label{WClemma:i3} $\| \alpha_k^\mathcal{U} (a_{g,i}) - a_{kg,i}\| < \ep$, for all $i=1,\dots,m$, all $k \in K$ and all $g \in T_i\cap k^{-1}T_i$,
\item \label{WClemma:i4} $ \|  a_{g,i} \| < \ep$, for all $i=1,\dots,m$, and $g \in T_i$ such that $Kg\nsubseteq T_i$,
\item \label{WClemma:i1} $\Big\|1-\sum_{i=1}^m\sum_{g\in T_i}a_{g,i}\Big\|_{2,\tau}<\ep$, for all $\tau\in T(A)^\alpha$.
\end{enumerate}
\end{prop} 

\begin{proof}
We may assume that $\ep<1$. Choose $n\in\mathbb{N}$ with $n > 1/\ep$.
Without loss of generality, we 
assume that $K$ is symmetric and contains the identity of $G$.
Let $0<\delta<1-(1-\ep^2)^{1/2}$. By definition, any $(K^n, \delta)$-invariant non-empty finite set $S \subseteq G$ satisfies
	\[
	\Big \lvert \bigcap_{g \in K^n} g S \Big \rvert > (1 - \ep^2)^{1/2} \lvert S \rvert.
	\]

Use the weak Rokhlin property with respect to invariant traces to find 
$(K^n,\delta)$-invariant finite subsets $S'_1,\ldots,S'_m\subseteq G$
and pairwise orthogonal positive contractions $a'_{g,i}\in A^\mathcal{U}\cap A'$,
for $i=1,\ldots,m$ and $g\in S'_i$, such that
\begin{enumerate}[label=(\alph*)]
\item \label{item:aa} $\alpha_g^{\mathcal{U}}(a'_{h,i})=a'_{gh,i}$, 
for all $i=1,\ldots,m$, $g\in G$, and $h\in S'_i$ for which $gh\in S'_i$.
\item  \label{item:bb} $\tau\Big(1-\sum_{i=1}^m\sum_{g\in S'_i}a'_{g,i}\Big)<\delta$, for all $\tau\in T(A)^\alpha$.
\end{enumerate}
Upon translating the sets $S'_1,\ldots,S'_m$, we may without loss of generality 
assume that they all contain the identity $e$ of $G$. 

For every $i = 1, \dots, m$, we set $T^{(i)}_0 = \bigcap_{g \in K^n} gS_i'$. 
For $j=0 ,\ldots, n-1$, we inductively define
\[
T^{(i)}_{j+1} := K^{j+1} T^{(i)}_0 \setminus K^j T^{(i)}_0.
\]
By our choice of $\delta > 0$, we have
\begin{equation} \label{eq:Ti}
\big| T^{(i)}_0 \big| > (1-\ep^2)^{1/2}\big| S^\prime_i \big|.
\end{equation}
Moreover, given $i=1,\ldots,m$ and $k \in K$, we have
\begin{align} \label{eq:kTi}
kT^{(i)}_0 \subseteq T^{(i)}_1\sqcup T^{(i)}_0 \ \mbox{ and } \
kT^{(i)}_j \subseteq T^{(i)}_{j+1} \sqcup T^{(i)}_j \sqcup T^{(i)}_{j-1}=K^{j+1}T^{(i)}_0\setminus K^{j-2}T^{(i)}_0,
\end{align}
for $j=1,\ldots,n-1$. This is the case since for every $g \in T^{(i)}_j$ we have that $kg \in K^{j+1}T^{(i)}_0$, while if we had $kg \in K^{j-2}T^{(i)}_0$ then $g$ would belong to $K^{j-1}T^{(i)}_0$, as $K$ is symmetric, contradicting the membership of $g$ in $T^{(i)}_j$.

For $i=1,\ldots,m$, set $S_i =\bigsqcup_{j=0}^n T^{(i)}_j$. Then $S_i=K^nT^{(i)}_0$, and thus $S_i\subseteq S_i'$. Moreover, we set $T_i =\bigsqcup_{j=0}^{n-1} T^{(i)}_j$, and 
note that $KT_i\subseteq S_i$ by \eqref{eq:kTi}. 
For $j=0,\ldots,n-1$ and $g \in T^{(i)}_j$, define
\[
a_{g,i} = \Big(1 - \frac{j}{n}\Big)a'_{g,i}\in A^\mathcal{U}\cap A',
\]
and note that these are pairwise orthogonal positive contractions.
We claim that these elements satisfy the conditions in the statement. 

To verify \eqref{WClemma:i3}, let $i=1,\ldots,m$, let $k \in K$, and let 
$g\in T_i\cap k^{-1}T_i$. Suppose that $g \in T^{(i)}_j$ for some 
$j=0,\ldots, n-1$. By \eqref{eq:kTi}, we have 
\[kg \in \begin{cases} 
          T^{(i)}_j \sqcup T^{(i)}_{j+1} & \mathrm{ if } \ j=0,\\
          T^{(i)}_{j-1} \sqcup T^{(i)}_j \sqcup T^{(i)}_{j+1} & \mathrm{ if } \ j=1,\ldots,n-2,\\
          T^{(i)}_{j-1} \sqcup T^{(i)}_{j} & \mathrm{ if } \ j=n-1.
         \end{cases}\]
Thus, using condition \ref{item:aa} in the definition of weak Rokhlin property with respect to invariant traces, and the fact that $g,kg\in S'_i$, we get
	\begin{equation}\label{eq:fkg}
	\big\|\alpha_k^\mathcal{U}(a_{g,i}) - a_{kg,i} \big\| \le \frac{1}{n} < \ep,
	\end{equation}
as desired.
To check item \eqref{WClemma:i4}, let $i=1,\ldots,m$, 
let $g \in T_i$ and choose $k \in K$ with $kg \notin T_i$. Then $kg \in S_i\setminus T_i= T^{(i)}_n$, and by \eqref{eq:kTi} we deduce that $g\in T^{(i)}_{n-1}$. 
Thus, $\|a_{g,i}\|\leq \frac{1}{n}<\ep$ by construction.

Finally, to check \eqref{WClemma:i1}, let $\tau \in T(A)^\alpha$. 
Set $a=\sum_{i=1}^m \sum_{g\in T_i}a_{g,i}$ and $a'=\sum_{i=1}^m \sum_{g\in S'_i}a'_{g,i}$. 
Denote $\tau_i=\tau(a'_{e,i})$, and note that $\tau(a'_{g,i})=\tau(\alpha_g^{\mathcal{U}}(a'_{e,i}))=\tau_i$, for $i=1,\ldots,m$ and $g\in S'_i$, using condition \ref{item:aa} and the fact that $e\in S_i'$.
We have 
\begin{align*}\label{eqn:tauff'}
\tau(a)\ &=\ \sum_{j=0}^{n-1}\sum_{i=1}^m \big(1-\frac{j}{n}\big)|T_j^{(i)}|\tau_i\\
&\geq \ \sum_{i=1}^m |T_0^{(i)}|\tau_i\\
&\stackrel{\mathclap{\eqref{eq:Ti}}}{>} \ \sum_{i=1}^m (1-\ep^2)^{1/2} |S_i'|\tau_i \\
&= \ (1-\ep^2)^{1/2}\tau(a').
\end{align*}
Therefore,
\begin{align*}
\|1-a\|^2_{2,\tau}&=\tau\big((1-a)^2\big)\leq \tau(1-a)\leq 1-(1-\ep^2)^{1/2}\tau(a')\\
&\stackrel{\mathclap{\ref{item:bb}}}{\leq} 1-(1-\ep^2)^{1/2}(1-\delta)\leq \ep^2, 
\end{align*}
as desired.
\end{proof}

In the next result, we use the 
weak Rokhlin property with respect to invariant traces 
to suitably average order zero 
maps $M_d\to A_{\mathcal{U}}\cap A'$ coming from uniform McDuffness
of $A$, in order to conclude McDuffness with respect to invariant traces of the 
action. This averaging process causes some loss of
information in terms of tracial smallness: 
if the original order zero map $\psi$
satisfies $1-\psi(1)\in J_{T(A)}$, then the averaged
order zero map $\varphi$ will only satisfy $1-\varphi(1)\in J_{T(A)^\alpha}$. By \autoref{thm:main1}, this is often
good enough to obtain $\mathcal{Z}$-stability of the 
crossed product.

\begin{thm} \label{lemma:cpcoz}
Let $G$ be a countable, discrete, amenable group, let 
$A$ be a separable, unital
\cstar-algebra, and let $\alpha\colon G\to \mathrm{Aut}(A)$ be an action
with the weak Rokhlin property with respect to invariant traces. 
Suppose that $T(A)\neq \emptyset$, and that $A$ is uniform McDuff. Then $(A,\alpha)$
has the McDuff property with respect to invariant traces.
\end{thm}
\begin{proof}
Let $d\in\mathbb{N}$, let $\ep>0$, and let $K\subseteq G$ be a finite symmetric subset. 
We will prove that there exists a completely positive 
contractive order zero map $\psi \colon M_d \to A^\cU \cap A'$ such that
	\begin{enumerate}
	\item \label{claim1:a} $\| \alpha_k^\cU(\psi(x)) - \psi(x) \| < \ep \| x \|$, for all $x \in M_d$ and $k \in K$,
	\item \label{claim1:b} $\| 1 - \psi(1) \|_{2, \tau} < \ep$, for every $\tau \in T(A)^\alpha$.
	\end{enumerate}
With $\kappa\colon A_\cU \to A^\cU$ denoting the canonical
quotient map,
a standard diagonalization argument then yields a completely positive contractive order zero map 
$\Psi \colon M_d \to \left(A^\cU \cap A'\right)^{\alpha^\cU}$ such that $1 - \Psi(1) \in \kappa\left(J_{T(A)^\alpha}\right)$. 
Once $\Psi$ is constructed, we use \cite[Proposition~1.2.4]{Win_covii} and \autoref{lemma:surjective} in order to obtain a completely positive contractive order zero map $\varphi\colon M_d\to \left(A_\cU \cap A'\right)^{\alpha_\cU}$ making the following diagram commute

\begin{equation*}
\xymatrix{
 && \left(A_\cU \cap A'\right)^{\alpha_\cU} \ar[d]^-{\kappa} \\%
M_d \ar[rr]_-{\Psi} \ar[urr]^-{\varphi} && \left(A^\cU \cap A'\right)^{\alpha^\cU}.
}
\end{equation*}
One readily checks that $\varphi$ satisfies $1-\varphi(1)\in J_{T(A)^\alpha}$.
		
We now construct the map $\psi$. To start, it follows from uniform McDuffness of $A$ that 
there exists a unital $*$-homomorphism $\lambda \colon M_d \to A^\cU \cap A'$ (see \cite[Definition~4.2]{cetw}).
Applying \autoref{lma:WeddingCake}, find finite subsets $T_1,\ldots,T_m\subseteq G$, 
and pairwise orthogonal positive contractions $a_{g,i}\in A^\mathcal{U}\cap A'$, 
for $i=1,\ldots,m$ and $g \in T_i$,
such that 
\begin{enumerate}[label=(\roman*)]
\item \label{claim1:i} $\| \alpha_k^\mathcal{U} (a_{g,i}) - a_{kg,i}\| < \frac\ep 3$, for all $i=1,\dots,m$, all $k \in K$ and all $g \in T_i\cap k^{-1}T_i$,
\item \label{claim1:ii} $ \|  a_{g,i} \| < \frac \ep 3$, for all $i=1,\dots,m$, and $g \in T_i$ such that $Kg\nsubseteq T_i$,
\item \label{claim1:iii} $\Big\|1-\sum_{i=1}^m\sum_{g\in T_i}a_{g,i}\Big\|_{2,\tau}<\ep$, for all $\tau\in T(A)^\alpha$.
\end{enumerate}
Moreover, letting $M =\bigcup_{i=1}^m\bigcup_{g \in T_i\cup KT_i} \alpha_g^\cU(\lambda(M_d))$, we may assume that $a_{g,i}\in A^{\mathcal{U}}\cap M'$, for $i=1,\ldots,m$ and $g\in T_i$, using a standard reindexation argument.
 
Define $\psi\colon M_d\to A^\cU\cap A'$ by
	\[
	\psi(x)  = \sum_{i=1}^m \sum_{g \in T_i} a_{g,i} \alpha_g^\cU(\lambda(x)),
	\]
for all $x\in M_d$. Using that the elements $a_{g,i}$, for 
$i=1,\ldots ,m$ and $g \in T_i$, are pairwise orthogonal positive contractions which commute with elements of $M$, one 
easily sees that $\psi$ is a completely positive contractive order zero map.
	
We now verify the inequality in item \eqref{claim1:a}. Fix $x \in M_d$ and $k \in K$. Without loss of generality, we assume that $\|x\|\leq 1$. 
Using at the third step the fact that all summands are mutually orthogonal (to see that, recall that the $a_{g,i}$ commute with $M$), we get
	\begin{align*}
 \alpha_k^\cU(\psi(x))&=
 \sum_{i=1}^m \sum_{g \in T_i}\alpha_k^\mathcal{U}(a_{g,i}) \alpha_{kg}^\cU(\lambda(x))\\
 &\stackrel{\ref{claim1:ii}}{\approx_{\frac{\varepsilon}{3}}} \sum_{i=1}^m\sum_{g \in T_i\cap k^{-1}T_i } \alpha_k^\mathcal{U}(a_{g,i})\alpha_{kg}^\cU(\lambda(x))\\
&\stackrel{{\ref{claim1:i}}}{\approx_{\frac{\varepsilon}{3}}}
\sum_{i=1}^m\sum_{g \in T_i\cap k^{-1}T_i} a_{kg,i}\alpha_{kg}^\cU(\lambda(x))\\
&\stackrel{{\ref{claim1:ii}}}{\approx_{\frac{\varepsilon}{3}}} \psi(x).
\end{align*}

	Finally, the inequality in \eqref{claim1:b} is an immediate consequence of condition~\ref{claim1:iii}. 
\end{proof}

Combining the results of this section with those proved in the previous ones,
we obtain the following corollary.

\begin{cor}\label{cor:iwRp}
Let $G$ be a countable, discrete, amenable group, let $A$ be a simple, separable, stably finite, unital, nuclear, $\mathcal{Z}$-stable
\cstar-algebra, and 
let $\alpha \colon G \to \mathrm{Aut}(A)$ be an  
action with the weak Rokhlin property with respect to invariant traces. 
Then $A \rtimes_\alpha G$ is a simple,
separable, unital, nuclear, $\mathcal{Z}$-stable
$C^*$-algebra. 
\end{cor}
\begin{proof} This follows from the fact that $\mathcal{Z}$-stability implies uniform McDuffness (see for example  the proof of (i) $\Rightarrow$ (ii) in \cite[Proposition 5.12]{KirchRor_centr}), together with \autoref{lma:iwRpOuter}, 
 \autoref{lemma:cpcoz} and  \autoref{thm:main2}.
\end{proof}

\section{Uniform Rokhlin property and finitely many extreme invariant traces} \label{s.URP}

In this section we employ our results, particularly \autoref{cor:iwRp} and \autoref{thm:main2}, to expand the current knowledge on Question
\ref{qstA}, and provide a positive answer for various actions which are not covered in the existing literature (particularly in \cite{GHV, Wouters}). As shown in \autoref{cor:uRp},
our main class of examples
comes from dynamical systems $(A, \alpha)$ for which $\partial_e T(A)$ is compact and such that the induced action $G \curvearrowright \partial_e T(A)$
has Niu's \emph{uniform Rokhlin property} (see \cite[Definition~3.1]{Niu_comparison_2019} or \autoref{def:URP}). As already discussed in the introduction, this family of examples includes all actions $(A,\alpha)$ of amenable groups $G$
for which $\partial_e T(A)$ is compact and finite-dimensional, and such that  $G \curvearrowright \partial_e T(A)$ is free (Corollary \ref{cor:E}), as well as some cases
where $\partial_e T(A)$ has infinite covering dimension (Corollary \ref{cor:Z}).

In the second part of this section we show that our methods also entail $\mathcal{Z}$-stability of crossed products as in Question
\ref{qstA} for which the set of invariant traces is a finite-dimensional simplex.

\subsection{Uniform Rokhlin property}
We begin with some definitions.

\begin{df}[{\cite[Definition~8.1]{Ker_dimension_2020}}]
Let $G$ be a countable, discrete group acting on a compact metrizable space $X$. 
An \textit{open castle} is a finite collection $\{\left(S_i,B_i\right)\}_{i=1}^m$, where $S_i\subseteq G$ are finite subsets, called \emph{shapes}, and $B_i\subseteq X$ are open subsets, called \emph{bases}, such that the sets $s_i B_i$, for $s_i\in S_i$ and $i=1,\ldots,m$, are pairwise disjoint.
\end{df}

For an action $G\curvearrowright X$ of a discrete group 
$G$ on a compact space $X$, we denote by $M_G(X)$ the set of $G$-invariant regular Borel probability measures on $X$.

\begin{df} [{\cite[Definition~3.1]{Niu_comparison_2019}}] \label{def:URP}
Let $G$ be a countable, discrete, amenable group and let $X$ be a compact metrizable space. An action
$G \curvearrowright X$ is said to have the \emph{uniform Rokhlin property} if for every $\ep>0$ and every finite subset $K \subseteq G$, there exists an open castle 
$\{(S_i,B_i)\}_{i=1}^m$ with $(K,\ep)$-invariant shapes such that 
\[\sup_{\mu\in M_G(X)}\mu\Big(X\setminus\bigsqcup_{i=1}^m \bigsqcup_{g \in S_i} g B_i\Big) <\ep.\]
\end{df}



In \autoref{def:URP} we may equivalently require that the castle $\{(S_i, B_i)\}_{i=1}^m$ is such that also
the closed sets $g \overline B_i$ are pairwise disjoint for all $i =1,\dots, m$ and $g \in S_i$. We make this explicit 
in the following lemma.


\begin{lma} \label{lemma:URPclosed}
Let $G$ be a countable, discrete, amenable group acting on a compact metrizable space $X$, let $K\subseteq G$ be a finite subset, let 
$\varepsilon>0$, and let
$\{ (S_i, B_i) \}_{i=1}^m$ be an open castle with $(K,\ep)$ invariant
shapes such that 
	\[
	\sup_{\mu\in M_G(X)}\mu\Big(X \setminus \bigsqcup_{i=1}^m \bigsqcup_{g \in S_i} g B_i \Big) < \ep.\]
Then there exist open sets $U_i$ satisfying $\overline{U}_i\subseteq B_i$ for all $i=1,\ldots,m$, such that
	\[
	\sup_{\mu\in M_G(X)}\mu\Big(X \setminus \bigsqcup_{i=1}^m \bigsqcup_{g \in S_i} g U_i \Big) < \ep.\]
In particular, the closed sets $g \overline U_i$ 
are pairwise disjoint for all $i =1,\dots, m$ and 
$g \in S_i$. 
\end{lma}
\begin{proof}
Set
\[
\delta=\ep -\sup_{\mu\in M_G(X)}\mu\Big(X \setminus \bigsqcup_{i=1}^m \bigsqcup_{g \in S_i} g B_i \Big)>0.
\] 
For each $i=1,\ldots,m$, set $V_i=\bigsqcup_{s\in S_i}sB_i$. Use 
\cite[Proposition~3.4]{KerSza_almost_2020} to find open subsets
$W_i\subseteq V_i$, for $i=1,\ldots,m$, so that the sets
\[U_i = \bigcup_{s\in S_i} s^{-1}(W_i\cap sB_i)\subseteq B_i, \quad i=1,\dots,m\]
satisfy the desired properties; we omit the details.
\end{proof}



We can now relate the uniform Rokhlin property for the induced action
on $\partial_eT(A)$ to the weak Rokhlin property with respect to invariant traces. Compactness of $\partial_e T(A)$ plays a crucial role
in this transfer from commutative dynamics to actions on stably finite, simple $C^*$-algebras, thanks to the results in
\cite{ozawa_dix}. Indeed, for unital $C^*$-algebras $A$ such that $T(A)$ is a non-empty
Bauer simplex, one of the consequences of this seminal work is the existence of a unital embedding
\begin{equation} \label{eq:embedding}
\theta \colon C(\partial_e T(A))\to A^\cU\cap A',
\end{equation}
such that $\tau (\theta(f)) = f(\tau)$ for every $\tau \in \partial_e T(A)$ (see \cite[Section 6]{KirchRor_centr}, particularly Proposition 6.6,
for an explicit proof of the existence of such embedding).

If $\alpha$ is an action on $A$, we denote the induced action on $C(\partial_eT(A))$ by $\alpha^{**}$. Arguing as in the proof of \cite[Proposition 8.3]{GHV}, it is possible
to check that the embedding $\theta$ from \eqref{eq:embedding} is equivariant. We summarize these facts in the following proposition.

\begin{prop}\label{thm:Ozawa}
	Let $\alpha\colon G\to \mathrm{Aut}(A)$ be an action of a countable, discrete group $G$ on a separable, unital $C^*$-algebra such that $T(A)$ is a non-empty Bauer simplex. Then there exists a unital, equivariant embedding 
	\[\theta\colon(C(\partial_e T(A)), \alpha^{**})\to (A^\cU\cap A',\alpha^\cU).\]
\end{prop}

\begin{prop}\label{prop:uRpiwRp}
Let $G$ be a countable, discrete, amenable group, let $A$ be a unital, separable $C^*$-algebra such that
$\partial_eT(A)$ is compact,
and let $\alpha\colon G\to \Aut(A)$ be an action such that the induced action $G\curvearrowright \partial_eT(A)$
has the uniform Rokhlin property. Then $\alpha$ has the weak Rokhlin property with respect to invariant traces. 
\end{prop}
\begin{proof}
Let $\ep>0$ and let $K\subseteq G$ be a finite subset.
We write $X$ for $\partial_eT(A)$ throughout.
Use the uniform Rokhlin property for $G\curvearrowright X$ to find an open castle $\{(S_i, B_i)\}_{i=1}^m$ with $(K, \ep)$-invariant shapes, such that for every $\mu \in M_G(X)$ we have
\begin{equation*}
\mu\Big( \bigsqcup\limits_{i=1}^m \bigsqcup\limits_{s_i \in S_i} s_iB_i\Big) > 1-\ep.
\end{equation*}
Moreover, by \autoref{lemma:URPclosed}, we can assume that there are open sets $(B'_i)_{i=1}^m$ such that $\overline{B'_i}\subseteq B_i$, and
for all $\mu \in M_G(X)$ we have
\begin{equation} \label{eq:measure}
\mu\Big( \bigsqcup\limits_{i=1}^m \bigsqcup\limits_{g \in S_i} gB_i'\Big) > 1-\ep.
\end{equation}
	
For $i=1,\ldots,m$, let $f_i\colon X\to [0,1]$
be a continuous function such that $f_i \equiv 1$ on $\overline{B'_i}$ and $\mathrm{supp}(f_i) \subseteq B_i$. For $g\in S_i$, set $f_{g,i}=g\cdot f_i$. 
Then $\mathrm{supp}(f_{g,i})\subseteq g B_i$ and therefore the functions $f_{g,i}$ are positive contractions with pairwise 
disjoint supports.
By \autoref{thm:Ozawa} there exists a
unital, equivariant embedding 
$\theta\colon C(\partial_eT(A))\to A^\U\cap A'$.
We claim that the positive contractions $\theta(f_{g,i})$ satisfy the conditions in \autoref{df:iwRp}. 

Condition \eqref{item:wrp1} is satisfied by construction. In order to 
check \eqref{item:wrp2}, let $\tau\in T(A)^\alpha$ and denote by 
$\mu\in M_G(X)$ the measure that $\tau$ induces on 
$C(\partial_eT(A))$ via the equivariant embedding $\theta$.
Set $f=\sum_{i=1}^m\sum_{g\in S_i}f_{g,i}\in C(\partial_eT(A))$. 
Then 
\begin{align*}
\tau(\theta(f))\geq\mu(\{x\in X\colon f(x)=1\}) &\ge \mu\Big(\bigsqcup_{i=1}^m \bigsqcup_{g \in S_i} g B_i'\Big) \stackrel{\mathclap{\eqref{eq:measure}}}{>} \ 1 - \ep.
\end{align*}
This shows that $\tau(1-\theta(f))<\ep$, as desired.	
\end{proof}

Combining the proposition above with \autoref{cor:iwRp}, we obtain:

\begin{cor}\label{cor:uRp}
Let $G$ be a countable, discrete, amenable group, let $A$ be a simple,
separable, unital, nuclear, $\mathcal{Z}$-stable
\cstar-algebra such that $T(A)$ is a non-empty Bauer simplex, and 
let $\alpha \colon G \to \mathrm{Aut}(A)$ be an action such
that $G \curvearrowright \partial_e T(A)$ has the uniform Rokhlin property. Then $A \rtimes_\alpha G$ is a simple,
separable, unital, nuclear, $\mathcal{Z}$-stable
$C^*$-algebra. 
\end{cor}

\subsection{Finitely many extreme invariant traces}
In this part we consider the last class of examples covered by \autoref{thm:main2}, namely those dynamical systems as in Question
\ref{qstA} for which the set of invariant traces is a finite-dimensional simplex. The interesting part of \autoref{thm:main2.}, when compared to the results in \cite{GHV, Wouters},
is the lack of topological assumptions on $T(A)$, which are compensated by the strong requirements on $T(A)^\alpha$.
\begin{cor}\label{thm:main2.}
	Let $G$ be a countable, discrete, amenable group, let $A$ be a simple, separable, unital, nuclear $\mathcal{Z}$-stable
	\cstar-algebra, and let $\alpha \colon G \to\mathrm{Aut}(A)$ be an action such that the set of invariant traces is a finite-dimensional simplex. Then $(A,\alpha)$ has the McDuff property with respect to invariant traces. In particular,
	if $\alpha$ is also outer, then 
	$A \rtimes_\alpha G$ is $\mathcal{Z}$-stable.
\end{cor}
\begin{proof}
Suppose that $\partial_e(T(A)^\alpha) = \{\tau_1,\ldots,\tau_n\}$, and let $\sigma=\frac{1}{n}\sum\limits_{i=1}^{n}\tau_i$. Similar argument as in \cite[Proposition~2.3]{GHV} shows that $J_{T(A)^{\alpha}}= J_\sigma$, and thus the quotient $A_\cU / J_{T(A)^\alpha}$ is isomorphic to the ultrapower $\mathcal{N}^\cU$ of the tracial von Neumann algebra $\mathcal{N} := \pi_\sigma(A)''$, where $\pi_\sigma$ is the GNS representation
associated to $\sigma$. The action $\alpha$ canonically extends to an action $\alpha_\sigma$ on $\mathcal{N}$, which
is an hyperfinite type II$_1$ von Neumann algebra, since $A$ is nuclear. By the classification of actions of discrete amenable groups on semifinite, hyperfinite von Neumann algebras in \cite{ST} (see also \cite[Section 3]{sut} and \cite[Theorem 5.4]{GHV}), the action
$\alpha_\sigma$ absorbs the identity action on the hyperfinite II$_1$ factor $\mathcal{R}$. By \cite[Proposition 5.3]{GHV},
this is enough in order to grant the existence of a unital embedding $M_d \to (\mathcal{N}^\cU \cap \mathcal{N}')^{\alpha_\sigma^\cU}$, for every $d\geq 2$.
In other words, the dynamical system $(A,\alpha)$ has the uniform McDuff property with respect to invariant traces. 
\end{proof}

\providecommand{\bysame}{\leavevmode\hbox to3em{\hrulefill}\thinspace}
\providecommand{\MR}{\relax\ifhmode\unskip\space\fi MR }
\providecommand{\MRhref}[2]{%
  \href{http://www.ams.org/mathscinet-getitem?mr=#1}{#2}
}
\providecommand{\href}[2]{#2}

\end{document}